\documentclass[11pt,a4paper]{article}
\usepackage{amsfonts,amsmath,amsthm,amssymb,mathtools,blkarray,bbm,fullpage,enumitem}
\usepackage[hidelinks]{hyperref}
\usepackage[capitalize,nameinlink]{cleveref}
\usepackage[numbers,sort]{natbib}
\usepackage{thm-restate}

\usepackage{todonotes}

\newtheorem{theorem}{Theorem}[section]
\newtheorem{lemma}[theorem]{Lemma}

\newtheorem{proposition}[theorem]{Proposition}
\newtheorem{corollary}[theorem]{Corollary}
\newtheorem{conjecture}[theorem]{Conjecture}

\theoremstyle{definition}
\newtheorem{definition}[theorem]{Definition}

\newcommand\ab[1]{\lvert #1 \rvert}
\newcommand\bab[1]{\left\lvert #1 \right\rvert}
\newcommand\dd{\,\mathrm{d}}
\newcommand\inner[2]{\left\langle #1, #2 \right\rangle}
\newcommand{\norm}[1]{\lVert#1\rVert}

\newcommand\R{\mathbb{R}}

\DeclareMathOperator*\odotprod{\odot}

\newlist{lemenum}{enumerate}{1}
\setlist[lemenum]{label=(\alph*), ref=\thelemma(\alph*)}
\crefalias{lemenumi}{lemma} 
\crefname{figure}{Figure}{Figures}

\title{\vspace{-0.8cm}Counting subgraphs in locally dense graphs}
\date{}

\author{
Domagoj Brada\v{c}\thanks{Department of Mathematics, ETH Z\"urich, Z\"urich, Switzerland. Research supported in part by SNSF grant 200021-228014. Email: \textbf{\{domagoj.bradac, benjamin.sudakov\}@math.ethz.ch}.}
\and 
Benny Sudakov\footnotemark[1]
\and
Yuval Wigderson\footnotemark[2]\thanks{Institute for Theoretical Studies, ETH Z\"urich, Z\"urich, Switzerland. 
   Supported by Dr.\ Max R\"{o}ssler, the Walter Haefner Foundation, and the ETH Z\"{u}rich Foundation. Email: {\textbf{yuval.wigderson@eth-its.ethz.ch}}.}
}

\begin{document}
\maketitle

\begin{abstract}
    A graph $G$ is said to be $p$-locally dense if every induced subgraph of $G$ with linearly many vertices has edge density at least $p$. A famous conjecture of Kohayakawa, Nagle, R\"odl, and Schacht predicts that locally dense graphs have, asymptotically, at least as many copies of any fixed graph $H$ as are found in a random graph of edge density $p$.

    In this paper, we prove several results around the KNRS conjecture. First, we prove that certain natural gluing operations on $H$ preserve this property, thus proving the conjecture for many graphs $H$ for which it was previously unknown. Secondly, we study a stability version of this conjecture, and prove that for many graphs $H$, approximate equality is attained in the KNRS conjecture if and only if the host graph $G$ is quasirandom. Finally, we introduce a weakening of the KNRS conjecture, which requires the host graph to be nearly degree-regular, and prove this conjecture for a larger family of graphs. Our techniques reveal a surprising connection between these questions, semidefinite optimization, and the study of copositive matrices.
\end{abstract}

\section{Introduction}
\subsection{Background}
Many of the most basic questions in extremal graph theory ask to understand, for a fixed graph $H$, how many copies of $H$ can appear in a large graph $G$ with certain constraints. For example, Mantel's theorem states that if $H=K_3$, then $G$ contains \emph{zero} copies of $H$ only if $G$ has at most $\lfloor v(G)^2/4\rfloor$ edges.

A convenient way of capturing the asymptotic nature of such problems is via the \emph{homomorphism density} $t(H,G)$, which is defined as the probability that a random function $f:V(H) \to V(G)$ is a graph homomorphism, that is, that $f$ maps every edge of $H$ to an edge of $G$. As ${v(G)} \to \infty$, asymptotically almost all homomorphisms are injective, so knowing $t(H,G)$ is essentially the same as knowing how many copies of $H$ there are in $G$. In this language, Mantel's theorem can be stated as $t(K_3,G)=0$ only if $t(K_2,G) \leq \frac 12$. More generally, the Erd\H os--Stone--Simonovits theorem \cite{MR0018807,MR0205876} states that $t(H,G)=0$ only if $t(K_2,G)\leq 1-\frac{1}{\chi(H)-1}+o(1)$, where the $o(1)$ term tends to $0$ as ${v(G)} \to \infty$ (for fixed $H$).

Even more generally, one could ask, for a fixed value of $p=t(K_2,G)$, what is the asymptotic minimum value of $t(H,G)$? 
In general, the answer to this question is extremely complicated \cite{MR2433944,MR2737279,MR3549620}. 
However, for certain bipartite graphs $H$, this problem has a very clean answer. For example, two applications of the Cauchy--Schwarz inequality imply that $t(C_4,G)\geq p^4$ for any graph $G$ satisfying $t(K_2,G)=p$. Moreover, an elementary computation shows that this bound is asymptotically tight if $G$ is a random graph of edge density $p$. A famous conjecture of Sidorenko \cite{sidorenko1993correlation} asserts that something similar should happen for \emph{all} bipartite $H$, namely that
\begin{equation}\label{eq:sidorenko}
t(H,G) \geq p^{e(H)}
\end{equation}
for all bipartite $H$, all $p \in [0,1]$, and all $G$ satisfying $t(K_2,G)=p$. Sidorenko's conjecture has been verified for many natural classes of bipartite graphs (see e.g.\ \cite{MR4237083,MR3893193,MR2738996,MR3456171,MR2607540}), but remains wide open in general. Note that again, Sidorenko's conjecture is best possible if true, since a random graph $G$ of edge density $p$ satisfies $t(H,G)=p^{e(H)}+o(1)$ a.a.s. 

Moreover, there is a sort of converse to the last statement. Following Chung--Graham--Wilson \cite{chung1989quasi}, let us say that a graph $G$ is \emph{$(p,\delta)$-quasirandom} if every set $S \subseteq V(G)$ contains $\frac p2 \ab S^2 \pm \delta v(G)^2$ edges. As proved by Thomason \cite{MR0930498} and Chung--Graham--Wilson \cite{chung1989quasi}, this condition is equivalent (up to a polynomial change in $\delta$) to many other natural notions of ``random-like'' behavior in graphs. In particular, one famous result of Chung--Graham--Wilson is that if $t(K_2,G)=p$ and $t(C_4,G) \leq p^4 + o(1)$, then $G$ is $(p,o(1))$-quasirandom. The \emph{forcing conjecture} (see \cite{MR2080111}) states that, again, such behavior should be more general: for any bipartite $H$ that is not a forest, we have that $t(H,G) = p^{e(H)}+o(1)$ only if $G$ is $(p,o(1))$-quasirandom. It is known (see e.g.\ \cite{MR2738996}) that the forcing conjecture implies Sidorenko's conjecture.

Additionally, the assumption that $H$ is bipartite is crucial for a statement like Sidorenko's conjecture to hold, since $t(H,K_{n,n})=0$ whenever $H$ is non-bipartite, despite the fact that $t(K_2,K_{n,n}) = \frac 12$. Nonetheless, it is tempting to wonder if, by imposing some extra condition on $G$, one can obtain a sensible statement, along the lines of \eqref{eq:sidorenko}, which holds for all $H$. Such a statement was conjectured by Kohayakawa, Nagle, R\"odl, and Schacht \cite{MR2595699}. To state it we first need the following definition.

\begin{definition} \label{def:locally-dense}
    A graph $G$ is called \emph{$(p,\delta)$-locally dense} if, for every $S \subseteq V(G)$ with $\ab S \geq \delta \ab {V(G)}$, we have $e(S) \geq p \frac{\ab S^2}2$.
\end{definition}
One can think of this condition as a ``one-sided'' version of quasirandomness---it does not assert that all large vertex sets have edge density roughly $p$ (as in quasirandomness), but only that all such sets have edge density \emph{at least} $p$. Simple examples, such as the disjoint union of two cliques, show that this is a strictly weaker condition than quasirandomness. Locally dense graphs and hypergraphs have been extensively studied and used over the years, see e.g.\ \cite{MR1276825,MR3666677,MR2595699,MR3171777,MR4201799,MR4608432,
MR0698654,
MR4111729,
MR3474967,
MR3848224,
MR3790065}. With this definition, we can state the conjecture of Kohayakawa--Nagle--R\"odl--Schacht \cite{MR2595699}.
\begin{conjecture}[Kohayakawa--Nagle--R\"odl--Schacht \cite{MR2595699}]\label{conj:KNRS}
    For every graph $H$ and all $p,\varepsilon \in (0,1)$, there exists some $\delta>0$ such that the following holds. If $G$ is a $(p,\delta)$-locally dense graph, then $t(H,G) \geq p^{e(H)}-\varepsilon$.
\end{conjecture}
Informally\footnote{Shortly, we will turn to the language of graphons, where these informal asymptotic notions can be made rigorous.}, this conjecture states that if $G$ is a $(p,o(1))$-locally dense graph, then $t(H,G) \geq p^{e(H)}-o(1)$ for any fixed $H$. For brevity, we will say that a graph $H$ is \emph{KNRS} if \cref{conj:KNRS} holds for $H$. Despite a great deal of interest, only a relatively restricted class of graphs are known to be KNRS, which we now briefly summarize.

First, if $H$ is a bipartite graph satisfying Sidorenko's conjecture, then $H$ is KNRS. Indeed, Sidorenko's conjecture is equivalent to \cref{conj:KNRS} in case $\varepsilon=0$ and $\delta=1$. As observed by Kohayakawa, Nagle, R\"odl, and Schacht \cite{MR2595699}, a simple inductive argument shows that all cliques (and, more generally, all complete multipartite graphs) are KNRS. They also noted that the main result from \cite{MR2864650} implies that all line graphs of hypercubes are KNRS. Next, answering a question raised by Kohayakawa--Nagle--R\"odl--Schacht, Reiher \cite{MR3171777} proved that all odd cycles are KNRS. Finally, Lee \cite{MR4201799} showed that all unicyclic graphs and all graphs obtained from a cycle by adding a chord are KNRS. Additionally, he showed that graphs obtained by certain gluing operations are KNRS, and in particular settled \cref{conj:KNRS} for all graphs on at most $5$ vertices.

\subsection{Our results}
In this paper, we prove a number of results around \cref{conj:KNRS}. Our first main result shows that certain ``highly symmetric'' graphs, which are obtained from smaller KNRS graphs by a certain gluing operation, are themselves KNRS. 
\begin{restatable}{definition}{gluingdef}
     Let $H_1, H_2$ be two graphs, let $I$ be an independent set in $H_1$ and let $a$ be a vertex in $V(H_1) \setminus I$. We define the graph $H_1 \ltimes_I^a H_2$ as follows. Let $A_1, \dots, A_{v(H_2)}$ be copies of $H_1$ with the set $I$ identified and otherwise disjoint. Furthermore, place a copy of $H_2$ on the set of $v(H_2)$ vertices corresponding to $a$ in each of $A_1, \dots, A_{v(H_2)}.$
\end{restatable}

\begin{figure}[h]
    \centering
    \includegraphics{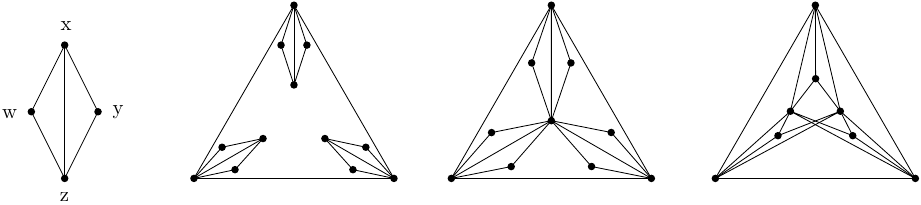}
    \caption{Let $D$ denote the first graph, which is a labelled diamond. The remaining graphs are, from left to right: $D \ltimes_{\varnothing}^x K_3, D \ltimes_{\{z\}}^x K_3$ and $D \ltimes_{\{y, w \}}^x K_3$.}
    \label{fig:products}
\end{figure}

\begin{theorem}\label{thm:gluing preserves KNRS}
    Suppose that $H_1, H_2$ are KNRS, $I$ is an independent set in $H_1$ and $a \in V(H_1) \setminus I$. Then $H_1 \ltimes_I^a H_2$ is KNRS.
\end{theorem}
This theorem already captures many of the known partial results about \cref{conj:KNRS}. For example, if $H_1$ is the star with $t$ leaves\footnote{An application of Jensen's inequality shows that all stars are Sidorenko, hence KNRS.}, $I$ is the set of its leaves and $a$ is the center, then $H_1 \ltimes_I^a H_2$ is obtained from $H_2$ by adding $t$ vertices connected to all vertices in $H_2$; thus repeated applications of \cref{thm:gluing preserves KNRS} imply that all complete multipartite graphs are KNRS. Similarly, applying it with $H_1 = P_\ell$ and $H_2=K_2$, where $a$ and $b$ are the two endpoints of the path\footnote{Blakley and Roy \cite{MR0184950} proved (in a different language) that all paths satisfy Sidorenko's conjecture, and hence are KNRS.} and $I=\{b\}$, we obtain that the odd cycle $C_{2\ell+1}$ is KNRS, recovering the result of Reiher \cite{MR3171777}. By adding an apex vertex (namely applying the theorem with $H_1=K_2$), we find that all wheels are KNRS. \cref{thm:gluing preserves KNRS} also implies that more complex graphs, such as the ones depicted in \cref{fig:products}, are KNRS.

The techniques used to prove \cref{thm:gluing preserves KNRS} are based on those developed by Reiher \cite{MR3171777} to prove that odd cycles are KNRS. However, we recast the entire problem, including a key lemma of Reiher's, in the more analytic setting of graphons, which we believe makes the proofs simpler and more transparent. Moreover, the true power of this perspective is even more pronounced in our subsequent results, which demonstrate surprising connections between \cref{conj:KNRS} and seemingly unrelated topics such as semidefinite optimization. To state our next main result, we make the following definition.
\begin{definition}
    A graph $H$ is called \emph{KNRS-forcing} if, for every $p, \alpha \in (0,1)$, there exist $\delta,\varepsilon>0$ such that the following holds. For every $(p, \delta)$-locally dense graph $G$, it holds that $t(H,G) > p^{e(H)} + \varepsilon$, unless $G$ is $(p,\alpha)$-quasirandom.
\end{definition}
Thus, a graph is KNRS-forcing if it is KNRS, and, moreover, the only locally dense graphs which asymptotically minimize $t(H,G)$ are themselves quasirandom. Just as Sidorenko's conjecture is strengthened by the forcing conjecture, we pose the following strengthening of \cref{conj:KNRS}.
\begin{conjecture}\label{conj:KNRS forcing}
    If $H$ is not a forest, then $H$ is KNRS-forcing.
\end{conjecture}
While we are not able to resolve \cref{conj:KNRS forcing} in general, we are able to prove it for certain natural graph families.
\begin{theorem}\label{thm:forcing intro}
    All cliques, cycles, and wheels are KNRS-forcing. Also, if $H_1$ is KNRS-forcing, then so is $H_1 \ltimes_I^a H_2$ for all independent sets $I \subseteq V(H_1),$ all $a \in V(H_1) \setminus I$, and all KNRS graphs $H_2$.
\end{theorem}
In fact, we are able to prove that $H_1 \ltimes_I^a H_2$ is KNRS-forcing even in certain instances when $H_1$ is not; we defer the precise statement to \cref{thm:main gluing}. We remark that our proof that odd cycles are KNRS-forcing is closely related to the fact that the Schatten norms are strictly convex on the positive-semidefinite cone. In particular, the KNRS-forcing condition---which roughly says that $t(H,{}\cdot{})$ has a unique global minimum---can be deduced in this instance from the fact that a strictly convex function has a unique global minimum. We discuss this in greater detail later, but we believe that such connections are an important feature of our techniques.

Finally, we also pose a weakening of \cref{conj:KNRS} which we can, perhaps surprisingly, prove for a larger class of graphs. Let us say that an $n$-vertex graph $G$ is \emph{$(p,\delta)$-nearly regular} if all but at most $\delta n$ vertices of $G$ have degree in the interval $[(p-\delta)n,(p+\delta)n]$.

\begin{conjecture} \label{conj:regular-KNRS}
    For every graph $H$ and all $p, \varepsilon \in (0, 1),$ there exists some $\delta > 0$ such that the following holds. If $G$ is a $(p, \delta)$-locally dense $(p, \delta)$-nearly regular graph, then $t(H, G) \ge p^{e(H)} - \varepsilon.$
\end{conjecture}
We say $H$ is \emph{regular-KNRS} if \cref{conj:regular-KNRS} holds for $H$. In other words, to check that $H$ is regular-KNRS, it suffices to prove that $t(H,G) \geq p^{e(H)}-o(1)$ over all nearly regular locally dense graphs $G$. It is known that in the setting of Sidorenko's conjecture, adding such a restriction is immaterial: $H$ is Sidorenko if and only if $t(H,G)\geq p^{e(H)}$ for all $(p,\delta)$-nearly regular graphs $G$ with $t(K_2,G)=p$. A simple proof of a slightly weaker statement can be found in \cite[Lemma 3.3]{MR3456171}, and stronger statements are proved in \cite{CoRa,Szegedy}. Because of this, it is natural to expect that $H$ is KNRS if and only if it is regular-KNRS, but we have been unable to prove this. However, since being regular-KNRS is a weaker condition, we were able to verify it for many more graphs.
\begin{theorem}\label{thm:regular KNRS anthology}
    The following classes of graphs are regular-KNRS.
    \begin{itemize}
        \item The $\ell$-subdivision of a regular-KNRS graph, for any $\ell \geq 1$. In particular, all balanced subdivisions of cliques are regular-KNRS.
        \item Any graph obtained from a regular-KNRS graph by gluing forests to its vertices.
        \item Any generalized theta graph, consisting of an arbitrary number internally vertex-disjoint paths, of arbitrary lengths, between two fixed vertices.
        \item The graph 
        $H=\tikz[scale=.5, baseline=4pt, every node/.style={circle, draw, fill=black, inner sep=.8pt}]{\draw (0,0) node (a) {} -- (1,0) node (b) {} -- (1,1) node (c) {} -- (0,1) node (d) {} -- cycle; \draw (a) -- (-.5,.5) node {} -- (d); \draw (b) -- (1.5,.5) node {} -- (c);}$%
        , obtained by gluing a triangle to two opposite edges of $C_4$.
    \end{itemize}
\end{theorem}
In particular, this result implies that all unicyclic graphs and all graphs obtained from a cycle by adding a chord are regular-KNRS, so the examples above include the graphs which Lee \cite{MR4201799} proved are KNRS, as well as generalizations of them; however, in contrast to Lee, we are only able to prove the weaker condition of regular-KNRS. Additionally, $H_0$ was raised by Lee \cite{MR4201799} as the smallest example of a graph which is not known to be KNRS; again our techniques are not able to prove this, but are able to show the weaker condition of being regular-KNRS.

In the analytic perspective we take, the regular-KNRS problem turns out to be substantially more well-behaved. At a high level, the reason is that locally dense graphs correspond to \emph{copositive matrices}, a class of matrices which is very poorly behaved. However, if one imposes the almost regularity condition, one instead obtains a positive semidefinite matrix, and thus we are able to use the rich theory of PSD matrices to prove results such as \cref{thm:regular KNRS anthology}. In particular, certain operations which preserve the PSD cone, such as tensor product, Hadamard product, and matrix powers, correspond to useful operations on graphs. Because of this, we believe that \cref{conj:regular-KNRS} is a natural setting in which to work and to obtain partial results, with the hopes of ultimately using similar techniques to resolve the full \cref{conj:KNRS}.

The rest of this paper is organized as follows. In \cref{sec:prelim}, we collect a number of general lemmas that we will use in our proofs, mostly related to the translation of these problems to the language of graphons. Proofs of many of these lemmas, as well as a more in-depth discussion of graph limits, are deferred to \cref{sec:graphon appendix}. The proofs of \cref{thm:gluing preserves KNRS,thm:forcing intro} are given in \cref{sec:gluing}. We prove results about regular-KNRS graphs, including \cref{thm:regular KNRS anthology}, in \cref{sec:regular-KNRS}. Finally, in \cref{sec:negative}, we collect a few examples which rule out natural avenues towards proving \cref{conj:KNRS}.

\section{Preliminaries}\label{sec:prelim}
In this section, we collect a few preliminary definitions and results that we will need in the proofs of our main theorems. 
We begin by introducing a weaker forcing notion, which will be useful in certain of our proofs, and which is interesting in its own right. 
\begin{definition}
    A graph $H$ is called \emph{density-forcing} if, for every $p,\alpha \in (0,1)$, there exist $\delta,\varepsilon>0$ such that the following holds. If $G$ is a $(p,\delta)$-locally dense graph with at least $(p+\alpha){v(G)}^2/2$ edges, then $t(H, G) > p^{e(H)} + \varepsilon$.
\end{definition}
It is not hard to show that if $H$ is density-forcing, it is also KNRS. Thus, loosely speaking, a graph $H$ is density-forcing if it is KNRS and the asymptotic minimizers all have edge density $p$; this is weaker than being KNRS-forcing, which says that the asymptotic minimizers are $p$-quasirandom.

\subsection{Graphons}
For the rest of the paper, we will work in the setting of graphons, rather than of graphs, as this makes a number of analytic arguments substantially simpler. In this section we simply recall the basic definitions, and state \cref{lem:graphon equivalence}, which restates all of the questions we are studying in the language of graphons. All of the results we need about graphons are essentially standard, but we provide detailed proofs in \cref{sec:graphon appendix}. For a more thorough introduction to the theory of graph limits, see \cite{MR3012035} or \cite[Chapter 4]{MR4603631}.
\begin{definition}
    Let $(\Omega, \Sigma, \mu)$ be an atomless standard probability space. A \emph{kernel} on $\Omega$ is a bounded symmetric measurable function $W:\Omega \times \Omega \to \mathbb{R}$. A kernel $W \colon \Omega \times \Omega \to [0,1]$ is called a \emph{graphon}. 

    As is standard in functional analysis, we will view two kernels that agree almost everywhere as the same object.
\end{definition}
Usually when working with graphons (or kernels), one just sets $\Omega$ to be the interval $[0,1]$ endowed with Lebesgue measure, which is without loss of generality\footnote{In fact, a graphon on an arbitrary probability space can be converted into one on $[0,1]$ with essentially no loss of generality; see \cite[Chapter 13]{MR3012035}.} as every atomless standard probability space is isomorphic to $[0,1]$. However, in some places in this paper, it will be convenient to work with more general probability spaces. When there is no confusion, we will use $\ab U$ to denote the measure of a measurable subset $U \subseteq \Omega$, and write $\int_\Omega f \dd x$ rather than $\int_\Omega f \dd \mu(x)$ for the integral of a function $f:\Omega \to \R$ with respect to $\mu$.

We now define some quantities associated to graphons.
The first is a graphon analogue of the degree sequence of a graph.
\begin{definition}
    If $W$ is a graphon, its \emph{degree function} is the function $d_W:\Omega \to [0,1]$ defined by 
    \[
    d_W(x) \coloneqq \int_\Omega W(x,y) \dd y.
    \]
    $W$ is called \emph{$p$-regular} if $d_W =p$ a.e.
\end{definition}

The second is the notion of homomorphism density, which extends naturally from graphs to kernels.
\begin{definition}
    Given a graph $H$, its homomorphism density in a kernel $W$ is given by
    \begin{equation} \label{eq:hom-density}
    t_\mu(H,W) \coloneqq \int_{\Omega^{V(H)}} \prod_{uv \in E(H)} W(x_u, x_v) \cdot \prod_{v \in V(H)} \dd \mu(x_v).
    \end{equation}
    When the measure $\mu$ is clear from the context, we will often omit the subscript.
    Additionally, for distinct $v_1, \dots, v_k \in V(H)$ and any $x_1, \dots, x_k \in \Omega,$ we write
    \[ t(H, W \mid x_{v_1} = x_1, \dots, x_{v_k} = x_k) \coloneqq \int_{\Omega^{V(H) \setminus \{v_1, \dots, v_k\}}} \prod_{uv \in E(H)} W(x_u, x_v) \cdot \prod_{v \in V(H) \setminus \{v_1, \dots, v_k\}} \dd x_v, \]
    where on the right-hand side we plug in $x_i$ for every instance of $x_{v_i}$.
\end{definition}
To every graph $G$, one can associate a graphon $W_G$ (see \cref{def:W_G}) such that $t(H,W_G)=t(H,G)$ for any graph $H$, hence this really is a generalization of the homomorphism density of graphs (see e.g. \cite[Section 4.3]{MR4603631} for more details).

Similarly, the definition of a locally dense graph translates naturally to that of a locally dense graphon. Because graphons already capture the asymptotic nature of the problem, the parameter $\delta$ from \cref{def:locally-dense} does not appear in the following definition.
\begin{definition}\label{def:locally dense graphon}
    A graphon $W$ on $(\Omega,\Sigma,\mu)$ is called \emph{$p$-locally dense} if, for every $U \subseteq \Omega$ we have that
    \[
    \iint_{U \times U} W(x,y) \dd x\dd y \geq p \ab U^2.
    \]
\end{definition}
Locally dense graphons precisely capture the asymptotic theory of locally dense graphs (see \cref{lem:graphon LD} for a precise statement). In particular, the following result shows that all of the properties we are studying are equivalently captured by the behavior of homomorphism counts in locally dense graphons.

\begin{lemma}\label{lem:graphon equivalence}
    Let $H$ be a graph. 
    \begin{lemenum}
        \item $H$ is KNRS if and only if $t(H,W) \geq p^{e(H)}$ for every $p$ and every $p$-locally dense graphon $W$. \label{lemit:KNRS}
        \item $H$ is KNRS-forcing if and only if $t(H,W) > p^{e(H)}$ for every $p$ and every $p$-locally dense graphon $W$, unless $W=p$ a.e. \label{lemit:KNRS-forcing}
        \item $H$ is density-forcing if and only if $t(H,W) > p^{e(H)}$ for every $p$ and every $p$-locally dense graphon $W$ with $\iint_{\Omega \times \Omega} W > p$.\label{lemit:density-forcing}
    \end{lemenum}
\end{lemma}

\subsection{An infinitary version of Reiher's lemma}
Reiher \cite{MR3171777} proved a very important lemma, which has been used in many of the results on locally dense graphs;  we will also use it multiple times throughout this paper. Informally, it states that a locally dense graph also satisfies a ``fractional'' version of the locally dense property in \cref{def:locally-dense}. Formally, it states the following.
\begin{lemma}\label{lem:reiher}
    Let $G$ be a $(p,\delta)$-locally dense graph with $n$ vertices. Let $f:V(G) \to [0,1]$ be a function satisfying $\sum_{v \in V(G)}f(v) \geq \delta n$. Then
    \[
    \sum_{xy \in E(G)} f(x) f(y) \geq \frac p2 \left(\sum_{x \in V(G)}f(x)\right)^2 - n.
    \]
\end{lemma}
Note that if we let $f$ be the indicator function of a subset $S \subseteq V(G)$, then we precisely recover the condition in \cref{def:locally-dense}, apart from the lower-order error term. Naturally, there is an infinitary analogue of \cref{lem:reiher}. We use the standard notation $\norm f_1$ to denote the $L^1$ norm of $f:\Omega \to \R$, that is, $\norm f_1 \coloneqq \int_\Omega \ab{f(x)}\dd x$.
\begin{lemma}\label{lem:reiher graphon}
    Let $W$ be a graphon on $\Omega$. $W$ is $p$-locally dense if and only if for every bounded measurable function $f:\Omega \to [0,\infty)$, we have
    \begin{equation}\label{eq:reiher graphon}
    \iint_{\Omega \times \Omega} f(x) W(x,y) f(y) \dd x \dd y \geq p \norm f_1^2.
    \end{equation}
\end{lemma}

An immediate but important corollary of \cref{lem:reiher graphon} is that the property of being locally dense is invariant under changing the probability measure on $\Omega$. To state this result properly, we need to introduce a bit of extra notation. As we will now be dealing with different measures on $\Omega$, all of our integrals will be written with respect to a given measure, e.g.\ we will write $\iint_{\Omega \times \Omega}W(x,y)\dd \mu(x)\dd\mu(y)$. Similarly, all $L^p$ norms will be computed with respect to a given measure, and we denote, for example, $\norm f_{1,\mu}$ for $\int_\Omega \ab{f(x)}\dd \mu(x)$. Let us also say that a graphon $W$ is \emph{$p$-locally dense with respect to $\mu$} if \cref{def:locally dense graphon} holds for the measure $\mu$, that is, if $\iint_{U \times U}W(x,y)\dd \mu(x)\dd\mu(y)\geq p\mu(U)^2$ for all measurable sets $U\subseteq \Omega$.

Given a weight function $w:\Omega \to [0,\infty)$ with $\int_\Omega w(x)\dd \mu(x)=1$, we can obtain a new probability measure $\nu_w$ on $\Omega$ by integrating $w$, namely the measure of a set $U$ is defined as $\nu_w(U) \coloneqq \int_U w(x)\dd \mu(x)$. Equivalently, $\nu_w$ can be defined by saying that for any bounded function $f:\Omega \to \R$, we have $\int_\Omega f(x)\dd \nu_w(x)=\int_\Omega f(x)w(x)\dd \mu(x)$. 
With this setup, we can now state the result that changing the measure preserves local density.
\begin{lemma} \label{lem:change-measure}
    Let $(\Omega,\Sigma,\mu)$ be an atomless standard probability space, and let $W$ be a graphon on $\Omega$. Let $w:\Omega \to [0,\infty)$ be a bounded weight function with $\int_\Omega w(x)\dd \mu(x)=1$, and let $\nu_w$ be the associated probability measure on $\Omega$. If $W$ is $p$-locally dense with respect to $\mu$, then it is $p$-locally dense with respect to $\nu_w$.
\end{lemma}

In our proofs, the following consequence of \cref{lem:change-measure} will be particularly useful. Both \cref{lem:change-measure,lem:count-with-weights} are proved in \cref{sec:graphon appendix}.
\begin{lemma} \label{lem:count-with-weights}
     Let $(\Omega,\Sigma,\mu)$ be an atomless standard probability space, and let $W$ be a graphon on $\Omega$. Let $w:\Omega \to [0,\infty)$ be a bounded weight function. If $H$ is a KNRS graph, then
     \[ \int_{\Omega^{V(H)}} \prod_{v \in V(H)} w(x_v) \prod_{uv \in E(H)} W(x_u, x_v) \prod_{v \in V(H)} \dd\mu(x_v) \ge \norm w_1^{v(H)} p^{e(H)}. \]
\end{lemma}

\subsection{Kernels as linear and bilinear operators}\label{sec:linear bilinear}
A kernel can be thought of as an infinite-dimensional generalization of a matrix, and just as matrices act as linear and bilinear operators, so do kernels. We collect here the important properties and definitions that we will need about such actions; a detailed introduction can be found in \cite[Section 7.5]{MR3012035}, and the functional analysis background can be found in, for example, \cite[Chapter 4]{MR2129625} or \cite[Chapter VI]{MR0751959}.

Given a kernel $W$ on $\Omega$ and a parameter $1 \leq p\leq \infty$, we can define a linear operator $T_W:L^p(\Omega) \to L^p(\Omega)$ by
\[
(T_W f)(x) \coloneqq \int_\Omega W(x,y) f(y)\dd y.
\]
The fact that $W$ is bounded and that $\Omega$ is a probability space implies that $T_Wf \in L^p$ whenever $f \in L^p$. A function $f \in L^2$ is called an \emph{eigenfunction} of $T_W$ with \emph{eigenvalue} $\lambda$ if $T_Wf = \lambda f$ a.e. It follows from standard Hilbert--Schmidt theory that $T_W$ is a self-adjoint compact operator on $L^2$, hence the spectral theorem implies that there exist real numbers $\lambda_1,\lambda_2,\dots$ and functions $f_1,f_2,\dots \in L^2$ such that each $f_i$ is an eigenfunction of $T_W$ with eigenvalue $\lambda_i$, and such that the set $\{f_i\}$ forms an orthonormal basis of $L^2$. In particular, this implies that there is a spectral decomposition of $W$ as
\[
W(x,y) \sim \sum_{i=1}^\infty \lambda_i f_i(x) f_i(y).
\]
Here, the symbol $\sim$ denotes that the series on the right-hand side may not converge for a.e.\ $x,y$, but that the right-hand side converges to $W(x,y)$ in $L^2$. We note for future reference that, although a priori the eigenfunctions $f_i$ are only in $L^2$, one can show that eigenfunctions associated to a non-zero eigenvalue are bounded; see \cite[Proposition 7.17]{MR3012035} for a proof. The following result, an analogue of one of the basic results in spectral graph theory, shows that the eigenvalues of a graphon can be used to compute the homomorphism count of any cycle.
\begin{theorem}[{\cite[(7.22)]{MR3012035}}] \label{thm:cycle-count}
    Let $W$ be a kernel and let $\lambda_1, \lambda_2, \dots, $ be its sequence of eigenvalues with $|\lambda_1| \ge |\lambda_2| \ge \dots.$ Then, for any $k \ge 2,$
    \[ t(C_k, W) = \sum_{i=1}^\infty \lambda_i^k. \]
\end{theorem}

Recall that a matrix $M \in \R^{n\times n}$ is called \emph{copositive} if $v^T Mv \geq 0$ for all vectors $v \in \R^n$ all of whose coordinates are non-negative.
Extending the definition of copositivity from matrices, let us say that a kernel $W$ is \emph{copositive} if, for every bounded $f:\Omega \to [0,\infty)$, we have
\[
\iint_{\Omega \times \Omega} f(x)W(x,y)f(y) \geq 0.
\]
More restrictively, we say that $W$ is \emph{positive semidefinite} if the same condition holds for every bounded $f:\Omega \to \R$, rather than only for bounded non-negative functions. 

Recall that $L^2$ is endowed with a natural inner product, given by $\inner fg = \int_\Omega f(x)g(x)\dd x$. With this notation, we can more concisely say that $W$ is copositive if $\inner f{T_W f}\geq 0$ for every bounded non-negative $f$, and that it is positive semidefinite if $\inner f{T_W f}\geq 0$ for every bounded real-valued $f$.
We will need the following two basic lemmas; the first is a generalization of a well-known fact about matrices, while the second connects locally dense graphons to copositive kernels. Proofs of both are given in \cref{sec:graphon appendix}.
\begin{lemma}\label{lem:PSD eigenvalues}
    A kernel $W$ is positive semidefinite if and only if all eigenvalues of $T_W$ are non-negative.
\end{lemma}

\begin{lemma}\label{lem:LD iff copositive}
    A graphon $W$ is $p$-locally dense if and only if the kernel $W-p$ is copositive.
\end{lemma}

In general, being positive semidefinite is a much stronger condition than being copositive, and the space of positive semidefinite kernels is much more well-behaved than the space of copositive kernels. Nonetheless, there is a simple condition which can be used to show that a copositive kernel is positive semidefinite.
\begin{lemma} \label{lem:positive in nullspace}
    Let $W$ be a  kernel. Suppose that there exists a bounded function $g:\Omega \to \R$ with $\inf_\Omega g > 0$ and $T_W g = 0$ a.e. If $W$ is copositive, then it is positive semidefinite.
\end{lemma}
\begin{proof}
    Let $f:\Omega \to \R$ be a bounded function. Since $\inf g > 0$, for sufficiently large $C \in \R$, we have $f+Cg \geq 0.$ Moreover, $f+Cg$ is bounded, so the copositivity of $W$ implies that
    \[ 0 \leq \inner{f+Cg}{T_W(f+Cg)} = \inner f{T_Wf} + 2C\inner{f}{T_W g} + C^2 \inner{T_W g}{T_W g} = \inner{f}{T_Wf}. \]
    As $f$ was arbitrary, $W$ is positive semidefinite.
\end{proof}
We will mostly use this lemma in the form of the following corollary. 
\begin{corollary}\label{cor:copositive-regular}
    Let $W$ be a $p$-locally dense, $p$-regular graphon. Then $W-p$ is positive semidefinite.
\end{corollary}
\begin{proof}
    If we let $g$ denote the constant $1$ function, then the $p$-regularity of $W$ is equivalent to saying that $T_W g= pg$. Hence $T_{W-p}g = 0$. As $W-p$ is copositive by \cref{lem:LD iff copositive}, \cref{lem:positive in nullspace} implies that $W-p$ is positive semidefinite.
\end{proof}
\cref{cor:copositive-regular} is, essentially, the reason why we are able to prove substantially more about regular-KNRS graphs than KNRS graphs. Indeed, since the space of PSD kernels is well-behaved, many arguments work for proving that a given graph is regular-KNRS but fail for proving that it is KNRS.

Given two kernels $W_1,W_2$, their \emph{operator product} $W_1 \circ W_2$ is defined by
\[
(W_1 \circ W_2)(x,y) \coloneqq \int_\Omega W_1(x,z)W_2(z,y)\dd z.
\]
    Note that $W_1\circ W_2$ is another kernel, as the boundedness of $W_1,W_2$ implies that $W_1 \circ W_2$ is bounded as well. This definition extends the standard definition of matrix multiplication, and is chosen so that $T_{W_1 \circ W_2} = T_{W_1} \circ T_{W_2}$, where $\circ$ on the right-hand side denotes function composition. We denote the \emph{operator power}, defined as the $k$-fold operator product of $W$ with itself, by $W^{\circ k}$. The main result we will need about operator products is that $W$ and $W^{\circ k}$ have the same eigenfunctions, and that the eigenvalues of $W^{\circ k}$ are given by $\{\lambda_i^k\}$, where $\{\lambda_i\}$ are the eigenvalues of $W$.

Another important operation is the \emph{tensor product}. Given two kernels $W_1,W_2$ on $\Omega$, their tensor product $W_1 \otimes W_2$ is a kernel on $\Omega \times \Omega$ defined by
\[
(W_1 \otimes W_2)((x_1,x_2),(y_1,y_2)) \coloneqq W_1(x_1,y_1)W_2(x_2,y_2).
\]
It is not hard to show, extending the well-known property of matrix tensor products, that the eigenvalues of $W_1 \otimes W_2$ are given by $\{\lambda_i \mu_j\}_{i,j=1}^\infty$, where $\{\lambda_i\},\{\mu_j\}$ are the eigenvalues of $W_1,W_2$, respectively.
The most useful property of the tensor product is that graph homomorphism densities are multiplicative over it, that is, for every graph $H$ and every pair of kernels $W_1,W_2$, we have
\begin{equation} \label{eq:hom-density-in-tensor}
t(H,W_1 \otimes W_2) = t(H,W_1)t(H,W_2).
\end{equation}
We denote the \emph{tensor power}, defined as the $k$-fold tensor product of $W$ with itself, by $W^{\otimes k}$.
\begin{lemma} \label{lem:walks-psd}
    If $W$ is a positive semidefinite kernel, then $W^{\circ k}, W^{\otimes k}$ are positive semidefinite for all $k\geq 1$.
\end{lemma}
\begin{proof}
    This follows immediately from \cref{lem:PSD eigenvalues} and the characterization of the eigenvalues of $W^{\circ k}$ and $W^{\otimes k}$.
\end{proof}
In addition to matrix powers and tensor products, another matrix operation that preserves positive-semidefiniteness is the Hadamard product. In the world of kernels, this corresponds to the pointwise product, namely given kernels $W_1,W_2$, we define the Hadamard product $W_1 \odot W_2$ by $(W_1 \odot W_2)(x,y) \coloneqq W_1(x,y) W_2(x,y)$. Since $W_1,W_2$ are measurable and bounded, so is $W_1 \odot W_2$, hence it is also a kernel. As with matrices, the Hadamard product preserves the positive semidefinite cone.
\begin{lemma}\label{lem:hadamard PSD}
    If $W_1,W_2$ are positive semidefinite kernels, then so is $W_1 \odot W_2$.
\end{lemma}
The analogous statement for matrices follows immediately from the fact that the tensor product preserves positive-semidefiniteness, since the Hadamard product of two matrices is a principal submatrix of their tensor product. It does not seem that one can make such an argument for kernels, so our proof of \cref{lem:hadamard PSD} is rather long and is given in \cref{sec:graphon appendix}.

\section{Gluing operations}\label{sec:gluing}
In this section, we prove that various graphs are KNRS and KNRS-forcing. We begin with a strengthening of Reiher's theorem \cite{MR3171777}, proving that all cycles are KNRS-forcing.

\begin{theorem}\label{thm:cycle forcing}
    For any $k \ge 3,$ the $k$-cycle is KNRS-forcing.
\end{theorem}
\begin{proof}
    If $k$ is even, the result was proved in the seminal paper of Chung, Graham and Wilson~\cite{chung1989quasi}. Hence, we will assume that $k$ is odd and write $k = 2\ell + 1.$ By \cref{lemit:KNRS-forcing}, it is enough to prove that for any $p$-locally dense graphon $W,$ it holds that $t(H, W) = p^{e(H)}$ if and only if $W = p$ a.e. So let $W$ be a $p$-locally dense graphon with $t(H, W) = p^{e(H)}.$

    First we show that $W$ is $p$-regular. Indeed, if not, then by a result of Sidorenko~\cite{sidorenko1993correlation}, we have $\norm{W^{\circ \ell}}_1 = t(P_{\ell}, W) > p^{\ell},$ where $P_\ell$ is the path with $\ell$ edges. In other words, if $W$ is not $p$-regular, we have strictly more $\ell$-paths than expected. Note that we may count $k$-cycles by counting pairs of $\ell$-paths starting from the same vertex and whose endpoints are joined by an edge. Formally,
    \begin{align*}
        t(C_k, W) &= \int_\Omega \left(\int_{\Omega \times \Omega} W^{\circ \ell} (x, y) W^{\circ \ell} (x, z) W(y, z) \dd y \dd z\right) \dd x\\
        &\ge \int_{\Omega} (d_{W^{\circ \ell}}(x))^2 p \dd x \ge \norm{W^{\circ \ell}}_1^2 p \ge p^{2\ell+1} = p^k,  
    \end{align*}    
    where in the first inequality we used \cref{lem:reiher graphon} with the weight function $w(y) = W^{\circ \ell}(x, y),$ in the second inequality we used Cauchy-Schwarz and in the third inequality we used that the $\ell$-path is Sidorenko.

    By assumption $t(C_k, W) = p^k$ so each of the above inequalities must be tight. In particular, $t(P_\ell, W) = p^\ell,$ which by a result of Sidorenko~\cite{sidorenko1993correlation}, implies that $W$ is $p$-regular.
    
    Let us write $W = p + M.$ Note that $\inner{M}{1} = 0$ so $1$ is an eigenfunction of $W$ with eigenvalue $p$. Let $1 = f_1, f_2, f_3, \dots, $ be eigenfunctions of $W$ with corresponding eigenvalues $p = \lambda_1, \lambda_2, \lambda_3, \dots$ such that $\{f_i\}_{i \in \mathbb{N}}$ forms an orthornomal basis. Note that for $i \ge 2,$
    \[ \inner{W}{f_i} = \inner{p}{f_i} + \inner{M}{f_i} = \inner{M}{f_i}, \]
    since $f_i$ is orthogonal to the all one function. Hence, $f_i$ is also an eigenfunction of $M$ with the same eigenvalue. By~\cref{cor:copositive-regular}, $M$ is positive semidefinite, which implies $\lambda_i \ge 0,$ for all $i \ge 1.$ Finally, by \cref{thm:cycle-count}, we have 
    \[ t(C_k, W) = \sum_{i=1}^\infty \lambda_i^k = p^k + \sum_{i=2}^\infty \lambda_i^k \ge p^k. \]
    Equality implies that $\lambda_i = 0$ for all $i \ge 2,$ and thus $W = p$ a.e. as claimed.
\end{proof}

We remark that there is an alternative way of viewing the proof of \cref{thm:cycle forcing}. For a kernel $W$ and an integer $k \geq 3$, let $f_k(W) = \sum_{i=1}^\infty \lambda_i^k$, where $\{\lambda_i\}_{i=1}^\infty$ are the eigenvalues of $W$. From \cref{thm:cycle-count}, we know that $f_k(W)=t(C_k,W)$. The function $f_k$, when $k$ is odd, is not a convex function on the space of all kernels, but it is strictly convex when restricted to the cone of PSD kernels; this is closely related to the statement that the Schatten $p$-norm, for $p>1$, is strictly convex. A basic fact of convex optimization is that strictly convex functions have a unique local minimum (which is also the global minimum). This implies that $t(C_k,{}\cdot{})$ has a unique local and global minimum in the cone of PSD kernels, hence that a stability result holds for the count of $C_k$. As it is not hard to identify this minimum, this stability result in turn implies that $C_k$ is KNRS-forcing. 

While this connection may appear somewhat superficial---and we found it easiest to prove \cref{thm:cycle forcing} without speaking about strictly convex functions at all---we are hopeful that further such connections will be uncovered. In particular, it seems possible that techniques from convex and semidefinite optimization could be used to prove results around \cref{conj:KNRS}. 

We now recall the key definition of our gluing operation (see Figure~\ref{fig:products} for an illustration).

\gluingdef*

\begin{theorem}\label{thm:main gluing}
    Suppose that $H_1, H_2$ are KNRS, $I$ is and independent set in $H_1$ and $a \in V(H_1) \setminus I$. Then $H_1 \ltimes_I^a H_2$ is KNRS. Moreover, if $H_1$ is KNRS-forcing or $H_2$ is density-forcing and $H_1 \ltimes_I^a K_2$ is KNRS-forcing, then $H_1 \ltimes_I^a H_2$ is KNRS-forcing.
\end{theorem}
\begin{proof}
    We first prove that $H_1 \ltimes_I^a H_2$ is KNRS and the forcing statement will follow by examining the equality case. Note that $e(H_2 \ltimes_I^a H_2) = v(H_2) e(H_1) + e(H_2)$. Let $W$ be a $p$-locally dense graphon. Let $k = |I|$ and let us label the vertices of $I$ by $1, \dots, k$. For $z_1, \dots, z_k, y \in \Omega,$ let $w_{z_1, \dots, z_k}(y) = t(H_1, W \mid x_1 = z_1, \dots, x_k = z_k, x_a = y),$ that is, $w_{x_1,\dots, x_k}(y)$ counts the number of copies of $H_1$ with vertex $i$ embedded into $x_i$ for $i \in [k]$ and vertex $a$ embedded into $y.$ Observe that
    \begin{align*}
        t(H_1 \ltimes_I^a H_2, &W \mid x_1 = z_1, \dots, x_k = z_k)\\ &= \int_{\Omega^{V(H_2)}} \prod_{v \in V(H_2)} w_{z_1, \dots, z_k}(x_v) \prod_{uv \in E(H_2)} W(x_u, x_v) \prod_{v \in V(H_2)} \dd x_v\\
        &\ge \norm{w_{z_1, \dots, z_k}}_1^{v(H_2)} p^{e(H_2)},
    \end{align*}
    where in the inequality we used \cref{lem:count-with-weights} and the assumption that $H_2$ is KNRS.
    Since $H_1$ is KNRS,
    \[ t(H_1, W) = \int_{\Omega^k} \norm{w_{z_1, \dots, z_k}}_1 \prod_{i=1}^k \dd z_i \ge p^{e(H_1)}. \]
    Integrating over the choice of $z_1, \dots, z_k,$ applying~\cref{lem:count-with-weights} and Jensen's inequality for the function $x^{v(H_2)}$ on $[0, \infty),$ we have
    \begin{align*}
        t(H_1 \ltimes_I^a H_2, W) &= \int_{\Omega^k} t(H_1 \ltimes_I^a H_2, W \mid x_1 = z_1, \dots, x_k = z_k) \prod_{i=1}^k \dd z_i\\ &\ge \int_{\Omega^k} \norm{w_{z_1, \dots, z_k}}_1^{v(H_2)} p^{e(H_2)} \prod_{i=1}^k \dd z_i
        \ge t(H_1, W)^{v(H_2)} p^{e(H_2)} \ge p^{v(H_2)e(H_1) + e(H_2)},      
    \end{align*}
    showing that $H_1 \ltimes_I^a H_2$ is KNRS.

    Finally, to show the forcing statements, suppose that $t(H_1 \ltimes_I^a H_2, W) = p^{v(H_2) e(H_1) + e(H_2)}.$ Then, in the above chain of inequalities, equality holds throughout.
    In particular, $t(H_1, W) = p^{e(H_1)}$. Since  $H_1$ is KNRS-forcing, we obtain that $W = p$ a.e. Hence $H_1 \ltimes_I^a H_2$ is also KNRS-forcing.
    
    Finally, let us show that if $H_2$ is density-forcing and $H_1 \ltimes_I^a K_2$ is KNRS-forcing, then $H_1 \ltimes_I^a H_2$ is KNRS-forcing. In this case, we have
    \begin{align}
        \notag t(H_1 \ltimes_I^a H_2, W \mid x_1 = z_1, \dots, x_k = z_k) = \norm{w_{z_1, \dots, z_k}}_1^{v(H_2)} p^{e(H_2)},\\ \text{ for almost all } z_1, \dots, z_k \text{ such that } \norm{w_{z_1, \dots, z_k}}_1 > 0.  \label{eq:count-from-x}
    \end{align}
    Using~\eqref{eq:count-from-x} and the equality case of Jensen's inequality, it follows that
    \begin{equation} \label{eq:one-norms-almost-all-x}
        \norm{w_{z_1, \dots, z_k}}_1 = p^{e(H_1)}, \text{ for almost all } z_1, \dots, z_k.
    \end{equation}
    Consider $z_1, \dots, z_k$ for which $\norm{w_{z_1, \dots, z_k}}_1 = p^{e(H_1)}$ and $t(H_1 \ltimes_I^a H_2, W \mid x_1 = z_1, \dots, x_k = z_k) = p^{v(H_2) e(H_1) + e(H_2)}.$ Let $\nu_{z_1, \dots, z_k}(y) \coloneqq w_{z_1, \dots, z_k}(y) / \norm{w_{z_1, \dots, z_k}}_1$ and note that $\nu_{z_1, \dots, z_k}(y)$ is a probability measure on $W$. Observe that 
    \begin{align*}
        t(H_1 \ltimes_I^a H_2, &W, \mid x_1 = z_1, \dots, x_k = z_k)\\
        &= \int_{\Omega^{V(H_2)}} \prod_{v \in V(H_2)} w_{z_1, \dots, z_k}(x_v) \prod_{uv \in E(H_2)} W(x_u, x_v) \prod_{v \in v(H_2)} \dd \mu (x_v) \\
        &= \norm{w_{z_1, \dots, z_k}}_1^{v(H_2)} \cdot t_{\nu_{z_1, \dots, z_k}}(H_2, W),
    \end{align*}
    implying that $t_{\nu_{z_1, \dots, z_k}}(H_2, W) = p^{e(H_2)}.$ By~\cref{lem:change-measure}, the graphon $W$ is $p$-locally dense with respect to the measure $\nu_{z_1, \dots, z_k}$. Since $H_2$ is density-forcing, we have
    
    \begin{align*}
        p &= \int_{\Omega \times \Omega} W(y, y') \dd \nu_{z_1, \dots, z_k}(y) \dd \nu_{z_1, \dots, z_k}(y')\\
        &= \frac{1}{\norm{w_{z_1, \dots, z_k}}_1^2} \int_{\Omega \times \Omega} w_{z_1, \dots, z_k}(y) w_{z_1, \dots, z_k}(y') W(y, y') \dd \mu(y) \dd\mu(y')\\
        &= \frac{1}{\norm{w_{z_1, \dots, z_k}}_1^2} t(H_1 \ltimes_I^a K_2, W \mid x_1 = z_1, \dots, x_k = z_k)\\
        &= p^{-2e(H_1)} t(H_1 \ltimes_I^a K_2, W \mid x_1 = z_1, \dots, x_k = z_k).
    \end{align*}
    In other words, for almost all $z_1, \dots, z_k$, we have $t(H_1 \ltimes_I^a K_2, W \mid x_1 = z_1, \dots, x_k = z_k) = p^{2e(H_1) + 1},$ so
    \[ t(H_1 \ltimes_I^a K_2, W) = \int_{\Omega^k} t(H_1 \ltimes_I^a K_2, W \mid x_1 = z_1, \dots, x_k = z_k) \prod_{i=1}^k \dd z_i = p^{2e(H_1) + 1} = p^{e(H_1 \ltimes_I^a K_2)} . \]
    Because $H_1 \ltimes_I^a K_2$ is KNRS-forcing, we obtain that $W = p$ a.e., so $H_1 \ltimes_I^a H_2$ is KNRS-forcing as well.
\end{proof}

As a special case of \cref{thm:main gluing}, if $H_1$ is KNRS-forcing, then $H_1 \ltimes_I^a K_2$ is KNRS-forcing as well.
The converse, however, need not hold, e.g.\ $K_2$ is clearly not KNRS-forcing, but denoting by $x, y$ the endpoints of $K_2,$ the graph $K_2 \ltimes_{\{x\}}^y K_2 \cong K_3$ is KNRS-forcing.

\begin{corollary}
    For any $k \ge 3,$ the $k$-clique and the $k$-wheel are KNRS and KNRS-forcing.
\end{corollary}
\begin{proof}
The proof of the first statement is by induction on $k$. 
By \cref{thm:cycle forcing},
$K_3$ is KNRS and KNRS-forcing and assume, by induction, that $K_{k-1}$ is KNRS and KNRS-forcing.
Let $x,y$ be the two endpoints of $K_2$. 
We have that $K_k \cong K_2 \ltimes_{\{x\}}^y K_{k-1}$ for all $k$, showing by the previous theorem that $K_k$ is KNRS. Moreover, as $K_2 \ltimes_{\{x\}}^y K_2 \cong K_3$ is KNRS-forcing and $K_{k-1}$ is KNRS-forcing by induction, we have that $K_k$ is KNRS-forcing. 

Similarly, the wheel $W_k$ can be constructed as $K_2 \ltimes_{\{x\}}^y C_k$. Since $C_k$ is KNRS and KNRS-forcing by \cref{thm:cycle forcing}, and since $K_2 \ltimes_{\{x\}}^y K_2 \cong K_3$ is KNRS-forcing, we conclude that $W_k$ is KNRS and KNRS-forcing.
\end{proof}

\section{On the regular-KNRS conjecture}\label{sec:regular-KNRS}
In this section we prove \cref{conj:regular-KNRS} for a family of graphs which includes many graphs for which \cref{conj:KNRS} is not known. Note that if $W$ is a $p$-regular $p$-locally dense graphon, then by \cref{cor:copositive-regular}, $W - p$ is a $0$-regular, positive semidefinite kernel. For any graph $H,$ by expanding \eqref{eq:hom-density}, we obtain
\begin{equation}
    \label{eq:psd-expansion}
    t(H, W) = t(H, p + (W-p)) = \sum_{H' \subseteq H} p^{e(H) - e(H')} t(H', W-p),
\end{equation}
where $H'$ runs over all spanning subgraphs of $H$. This motivates the following definition.

\begin{definition}
    We say that a graph $H$ is \emph{PSD-nonnegative} if for any positive semidefinite $0$-regular kernel $W,$ it holds that $t(H, W) \ge 0.$ We say that $H$ is \emph{hereditarily PSD-nonnegative} if every spanning subgraph of $H$ is PSD-nonnegative.
\end{definition}

We remark that a stronger notion has been studied in the literature. Namely, a graph $H$ is said to be \emph{positive} if $t(H,W) \geq 0$ for all kernels $W$. Clearly, every positive graph is in particular PSD-nonnegative. The positive graph conjecture of Antol\'in Camarena--Cs\'oka--Hubai--Lippner--Lov\'asz \cite{MR3425981} predicts a precise characterization of positive graphs (see \cite{CoLeVe,MR4444712} for recent progress on this conjecture). However, for our purposes we need the notion of hereditary PSD-nonnegativity, which does not seem to have a direct connection to the theory of positive graphs.

Equation~\eqref{eq:psd-expansion} immediately implies the following.
\begin{proposition} \label{prop:all-subgraphs-psd}
    Every hereditarily PSD-nonnegative graph satisfies \cref{conj:regular-KNRS}.
\end{proposition}
\begin{proof}    
    Suppose that $H$ is a hereditarily PSD-nonnegative graph. Let $W$ be a $p$-regular $p$-locally dense graphon and let $W_0 \coloneqq W - p.$ Note that by \cref{cor:copositive-regular}, $W_0$ is $0$-regular and positive semidefinite. By \eqref{eq:psd-expansion}, we have
    \begin{align*}
        t(H, W) &= \sum_{H' \subseteq H} p^{e(H) - e(H')} t(H', W_0) = p^{e(H)} + \sum_{H' \subseteq H, E(H') \neq \varnothing} p^{e(H) - e(H')} t(H', W_0) \ge p^{e(H)},
    \end{align*}
    where in the inequality we used the assumption that all spanning subgraphs of $H'$ are PSD-nonnegative.
\end{proof}

\begin{lemma} \label{lem:glue-psd}
    If $W$ is a positive semidefinite kernel and $W'$ is an arbitrary kernel on the same probability space $(\Omega, \Sigma, \mu)$, then the kernel $W' \circ W \circ W'$ is positive semidefinite.
\end{lemma}
\begin{proof}
    Let $f \colon \Omega \rightarrow [0, \infty)$ be an arbitrary bounded measurable function. Setting $\Phi(x') = \int_\Omega f(x) W(x, x') \dd x,$ we have
    \begin{align*}
        \int_{\Omega \times \Omega} \!\!f(x) (W' \!\circ W \circ W')(x, y) f(y) \dd x \dd y = &\int_{\Omega^4}\! f(x) W'(x, x') W(x', y') W'(y', y) f(y) \dd x \dd x' \dd y' \dd y\\= &\int_{\Omega \times \Omega} \Phi(x') W(x', y') \Phi(y') \dd x' \dd y' \ge 0,        
    \end{align*}
    where in the last inequality we used that $W$ is positive semidefinite. As $f$ was arbitrary, the statement follows.
\end{proof}

Given a kernel $W$ and a graph $H$ with two distinct vertices labelled $a$ and $b,$ we define a new kernel $t_{a, b}(H, W)$ by setting
\[ t_{a, b}(H, W)(x, y) = t(H, W \mid x_a = x, x_b = y). \]

For $s_1, \dots, s_t$ with $s_1 \ge 0$ and $s_2, \dots, s_t \ge 1,$ let $\Theta(s_1, \dots, s_t)$ be the graph consisting of two distinct vertices $a$ and $b$ and $t$ internally vertex disjoint paths between $a$ and $b$, with the $i$th path having $s_i$ internal vertices, where if $s_1 = 0,$ there is an edge between $a$ and $b.$ We call $\Theta(s_1, \dots, s_t)$ a \emph{generalized $\Theta$-graph}.

\begin{lemma} \label{lem:theta-graphs-psd}
    Let $s_1 \ge 0, s_2 \dots, s_t \ge 1,$ be arbitrary, denote $\Theta = \Theta(s_1, \dots, s_t)$ and let $a, b$ be the two distinguished vertices of $\Theta.$ Then, for any positive semidefinite kernel $W,$ the kernel $t_{a, b}(\Theta, W)$ is positive semidefinite.
\end{lemma}
\begin{proof}
    Indeed, observe that
    \[ t_{a, b}(\Theta, W)(x, y) = \prod_{i=1}^t W^{\circ s_i+1}(x, y), \]
    so we can write
    \[ t_{a, b}(\Theta, W) = \odotprod_{i=1}^t W^{\circ s_i + 1}. \]
    By \cref{lem:walks-psd}, $W^{\circ s_i + 1}$ is positive semidefinite for all $i \in [t]$ and by \cref{lem:hadamard PSD}, the Hadamard product of two positive semidefinite kernels is positive semidefinite, implying that $t_{a, b}(\Theta, W)$ is positive semidefinite.
\end{proof}

\begin{lemma} \label{lem:degree-one-psd}
    If $H$ has a vertex of degree one, then $H$ is PSD-nonnegative.
\end{lemma}
\begin{proof}
    In fact we prove that for any $0$-regular kernel $W,$ it holds that $t(H, W) = 0.$ Let $v$ be a vertex of degree one in $H,$ let $u$ be its unique neighbour and let $H' = H \setminus \{v\}.$ Then for a $0$-regular kernel $W,$ we have
    \[ t(H, W) = \int_{\Omega} t(H', W \mid x_u = x) \int_\Omega W(x, y) \dd y \dd x = 0, \]
    since $\int_\Omega W(x, y) \dd y = 0$ for almost all $x$.
\end{proof}

\begin{theorem} \label{thm:theta-graph}
    Any generalized $\Theta$-graph is hereditarily PSD-nonnegative and, in particular, it satisfies \cref{conj:regular-KNRS}.
\end{theorem}
\begin{proof}
    Let $H'$ be a subgraph of a generalized $\Theta$-graph. We need to show that $H'$ is PSD-nonnegative. Indeed, note that every spanning subgraph of a generalized $\Theta$-graph is itself a generalized $\Theta$-graph (potentially plus isolated vertices) or has a vertex of degree one. In the latter case, we are done by \cref{lem:degree-one-psd}, so assume that $H'$ is a generalized $\Theta$-graph plus isolated vertices. Homomorphism densities are multiplicative over disjoint unions, so we may remove the isolated vertices without changing $t(H',W)$. By \cref{lem:theta-graphs-psd}, $t_{a,b}(H', W)$ is positive semidefinite, so in particular,
    \[ t(H', W) = \int_{\Omega \times \Omega} t_{a, b}(H', W)(x, y) \dd x \dd y \ge 0, \]
    as needed.
\end{proof}

\begin{theorem} \label{thm:H0}
    Let $H$ be the graph with 6 vertices and 8 edges, where $V(H) = \mathbb{Z}_6$ and $E(H) = \{ \{i, i+1\} \, \colon \, 1 \le i \le 6\} \cup \{\{1,5 \} \{2, 4\} \}.$ $H$ is hereditarily PSD-nonngetive and thus it satisfies \cref{conj:regular-KNRS}.
\end{theorem}
\begin{proof}
    See Figure~\ref{fig:H0} for an illustration of $H$ and its labelling. We need to show that every spanning subgraph of $H$ is PSD-nonnegative. Let $H'$ be a spanning subgraph of $H$. It is easy to verify that $H'$ satisfies at least one of the following cases which we handle in order.
    \begin{itemize}
        \item $H'$ has a vertex of degree one. Then $H'$ is PSD-nonnegative by \cref{lem:degree-one-psd}.
        \item $H'$ is a generalized $\Theta$-graph. Then $H'$ is PSD-nonnegative by \cref{thm:theta-graph}.
        \item $H'$ is a disjoint union of two triangles. For any kernel $W$, $t(H', W) = t(K_3, W)^2 \ge 0.$
        \item $H'$ is two vertex disjoint triangles connected by a single edge. Let $x, y$ denote the two vertices of degree 3 in $H'.$ Let $f(x) = t(K_3, W \mid x_1=x),$ i.e. $f(x)$ counts triangles containing $x.$ Then $t(H', W) = \iint f(x) W(x, y) f(y) \ge 0,$ since $W$ is positive semidefinite.
        \item $H' \cong H.$ Denote $F = H[\{1, 2, 3, 4, 5\}] \setminus \{\{1,5\}\}$. Let $W_3 = t_{a, b}(K_3, W)$ where $a, b$ are two distinct vertices of a triangle. Note that a triangle is a generalized $\Theta$-graph so by \cref{lem:theta-graphs-psd}, $W_3$ is a positive semidefinite kernel. Additionally, observe that $t_{1, 5}(F, W) = W \odot W_3 \odot W,$ so $t_{1, 5}(F, W)$ is positive semidefinite by \cref{lem:glue-psd}. Finally, note that 
        $t_{1, 5}(H', W) = W \odot W^{\circ 2} \odot t_{1, 5}(F, W)$ which is positive semidefinite by \cref{lem:walks-psd} and \cref{lem:hadamard PSD}. In particular this implies that $t(H', W) = \int_{\Omega \times \Omega} t_{1, 5}(H', W)(x, y) \dd x \dd y \ge 0,$ as required.
        \qedhere
    \end{itemize}
\end{proof}

\begin{figure}[h]
    \centering
    \includegraphics{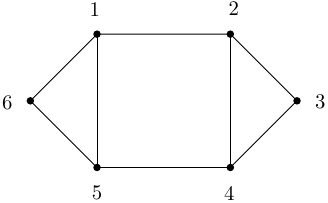}
    \caption{The graph $H$ from~\cref{thm:H0}.}
    \label{fig:H0}
\end{figure}
The above proof strategy can be used to prove~\cref{conj:regular-KNRS} for a slightly larger family of graphs, we have not included the most general statement for the sake of clarity of presentation.

Recall that the $\ell$-subdivision of a graph $H$ is obtained by replacing each edge by a path of length $\ell+1$.
\begin{proposition} \label{prop:subidivision-regular-KNRS}
    If a graph $H$ is regular-KNRS then for every $\ell \ge 1,$ the $\ell$-subdivision of $H$ is regular-KNRS.
\end{proposition}
\begin{proof}
    Let $H$ be regular-KNRS and let $W$ be a $p$-regular $p$-locally dense graphon. Let $H^{\circ \ell}$ denote the $\ell$-subdivision of $H.$ Crucially, note that $t(H^{\circ \ell}, W) = t(H, W^{\circ \ell+1}).$ By \cref{cor:copositive-regular}, $W_0 \coloneqq W - p$ is a $0$-regular positive semidefinite kernel. Observe that since $W_0$ is $0$-regular, $p \circ W_0 = W_0 \circ p = 0$ a.e. Hence, 
    \[ W^{\circ \ell + 1} = (p + W_0)^{\circ \ell + 1} = p^{\circ \ell + 1} + W_0^{\circ \ell + 1} = p^{\ell + 1} + W_0^{\circ \ell + 1}. \]
    Since $W_0$ is positive semidefinite, by \cref{lem:walks-psd}, $W_0^{\circ \ell + 1}$ is positive semidefinite and in particular, it is copositive. Furthermore, $W_0^{\circ \ell + 1}$ is $0$-regular. Indeed, for any $x \in \Omega,$ we have
    \[ \int_\Omega W_0^{\circ \ell + 1}(x, y) \dd y = \int_{\Omega^\ell} W_0(x, z_1) \prod_{i=1}^{\ell-1} W_0(z_i, z_{i+1}) \int_{\Omega} W_0(z_\ell, y) \prod_{i=1}^{\ell} \dd z_i \dd y. \]
    Since $W_0$ is $0$-regular, for almost every choice of $z_\ell,$ we have $\int_{\Omega} W_0(z_\ell, y) \dd y = 0$ and so $\int_\Omega W_0^{\circ \ell + 1}(x, y) \dd y = 0$ for all $x,$ implying that $W_0^{\circ \ell + 1}$ is $0$-regular. We conclude that $W^{\circ \ell + 1}$ is a $p^{\ell+1}$-locally dense, $p^{\ell+1}$-regular graphon. By the assumption, it follows that 
    \[ t(H^{\circ \ell}, W) = t(H, W^{\circ \ell + 1}) \ge (p^{\ell + 1})^{e(H)} = p^{e(H^{\circ \ell})},\]
    which completes the proof.
\end{proof}
We remark that in the setting of Sidorenko's conjecture, Conlon, Kim, Lee, and Lee \cite{MR3893193} proved a similar result, namely that the $1$-subdivision of any KNRS graph is Sidorenko.

Finally, we show that in the setting of \cref{conj:regular-KNRS}, we can use the so-called tensor power trick. We start with a lemma.

\begin{lemma} \label{lem:p-regular-tensor-square}
    Let $W$ be a $p$-regular, $p$-locally dense graphon. Then $W \otimes W$ is a $p^2$-regular, $p^2$-locally dense graphon.
\end{lemma}
\begin{proof}
    Let $W_0 \coloneqq W - p.$ We can write $W \otimes W$ as follows:
    \[ W \otimes W = (W_0 + p) \otimes (W_0 + p) = W_0 \otimes W_0 + W_0 \otimes p + p \otimes W_0 + p^2, \]
    where $p^2$ is the all-$p^2$ kernel on $\Omega \times \Omega.$
    By~\cref{cor:copositive-regular}, $W_0$ is $0$-regular and positive semidefinite. Clearly, the all-$p$ graphon is positive semidefinite. Recall that the tensor product of two positive semidefinite kernels is itself positive semidefinite and that the sum of two positive semidefinite kernels is positive semidefinite. Hence, $W \otimes W - p^2 = W_0 \otimes W_0 + W_0 \otimes p + p \otimes W_0$ is positive semidefinite and, in particular, it is copositive. By \cref{lem:LD iff copositive}, $W \otimes W$ is $p^2$-locally dense. It remains to verify that $W \otimes W$ is $p^2$-regular. Indeed, for almost all $(x, y) \in \Omega \times \Omega,$ we have
    \[ \int_{\Omega \times \Omega} (W \otimes W)((x, y), (x', y')) \dd x' \dd y' = \int_{\Omega \times \Omega} W(x, x') W(y, y') \dd x' \dd y' = p^2. 
    \qedhere
    \]
\end{proof}

\begin{proposition} 
    Let $H$ be a graph and suppose there is an absolute constant $c > 0$ such that $t(H, W) \ge c p^{e(H)}$ for every $p$-locally dense $p$-regular graphon $W.$ Then the same holds with $c=1,$ that is, $H$ is regular-KNRS.
\end{proposition}
\begin{proof}
    Let $H, c$ be as in the statement and let $W$ be an arbitrary $p$-locally dense $p$-regular graphon. Inductively applying \cref{lem:p-regular-tensor-square}, we obtain that for any $k \ge 0,$ the graphon $W^{\otimes{2^{k}}}$ is $p^{2^{k}}$-locally dense and $p^{2^{k}}$-regular. By \eqref{eq:hom-density-in-tensor} and the assumption, we have
    \[ t(H, W) = \left( t(H, W^{\otimes{2^{k}}}) \right)^{1 / 2^k} \ge \left(c p^{2^{k} \cdot e(H)}\right) ^{1 / 2^k} = c^{1/2^k} p^{e(H)}. \]
    Taking $k \rightarrow \infty,$ we obtain $t(H, W) \ge p^{e(H)},$ as claimed.
\end{proof}

\section{Negative results}\label{sec:negative}
In this section we collect counterexamples to several natural avenues of proving \cref{conj:KNRS}. 

\begin{proposition}\label{prop:counter-examples}
    For all $p \in (0, 1/3),$ there is a $p$-locally dense graphon $W$ satisfying the followng:
    \begin{enumerate}
        \item $t(K_2, W) = \frac{9p}{8}$.
        \item For any graphon $W'$ which is not $0$ a.e.,\ $W - W'$ is not $p$-locally dense. \label{prop:cannot-erase}
        \item $W \otimes W$ is not $p^2$-locally dense. \label{tensor-not-dense}
        \item For any $\ell \ge 5,$ $W^{\circ \ell}$ is not $p^\ell$-locally dense. \label{paths-not-tensor}
    \end{enumerate}
\end{proposition}
Let us briefly comment on the statement of \cref{prop:counter-examples}. The first two points imply the existence of a $(p, o(1))$-locally dense $n$-vertex graph with at least $(\frac{9p}{8} - o(1)) \binom{n}{2}$ edges such that if we remove any set of $\Omega(n^2)$ edges, the graph is no longer $(p, o(1))$-locally dense. In particular, this presents a difficulty in establishing that a KNRS graph $H$ is density forcing. Indeed, suppose such an example did not exist, so for any $(p, o(1))$ locally-dense graph $G$ with $(p + \varepsilon) \binom{n}{2}$ edges, there is a set $F$ of $\Omega(n^2)$ edges such that $G \setminus F$ is $(p, o(1))$-locally dense. It is easy to show that there are at least $\Omega(n^{v(H)})$ copies of $H$ using edges of $F$, and thus we would have $t(H, G) \ge \Omega(n^{v(H)}) + t(H, G \setminus F) \ge (1 + \delta) n^{v(H)} p^{e(H)}$ for some $\delta > 0$, which would imply that $H$ is density forcing.

The third point gives a counterexample to the direct way of trying to apply the tensor power trick to the setting of \cref{conj:KNRS}. We remark that in the setting of Sidorenko's conjecture, the tensor power trick allows one to reduce to the case when $G$ is nearly regular; thus the third point also shows the difficulty of proving that \cref{conj:regular-KNRS} implies \cref{conj:KNRS}. Finally, the fourth point illustrates why we cannot prove an analogue of \cref{prop:subidivision-regular-KNRS} without the assumption of $p$-regularity.

\begin{proof}[Proof of \cref{prop:counter-examples}]
    Let $W$ be the graphon arising from the matrix $A = (a_{ij})_{i,j}$ defined as
    \[ A = \begin{pmatrix}
        3p & 0 & 0 & 0\\
        0 & p & 2p & 2p\\
        0 & 2p & p & 2p\\
        0 & 2p & 2p & p\\
    \end{pmatrix} \]
    In other words, let $\Omega = [0,1)$ endowed with the Lebesgue measure and for $i \in [4],$ denote $I_i = [(i-1)/4, i/4).$ Then $W(x, y)$ is given by $a_{ij}$ where $x \in I_i, y \in I_j.$ It is simple to verify that $W$ is $p$-locally dense and that $t(K_2,W) = \frac{9p}{8}.$ 
    
    To see Property~\ref{prop:cannot-erase}, note that for $2 \le j \le 4,$ the set $[0,1] \setminus I_j$ has density exactly $p$ in $W,$ and every pair $(x, y) \in [0,1]^2$ is contained in one of these sets. 

    To prove Property~\ref{tensor-not-dense}, consider the set $(I_1 \times I_2) \cup (I_2 \times I_1) \cup (I_2 \times I_2) \cup (I_2 \times I_3) \subseteq [0,1]^2.$ This set has density $\frac{3}{4}p^2$ in $W \otimes W$.

    Finally, for Property~\ref{paths-not-tensor}, note that for $x \in I_1,$ we have $d_{W^{\circ \ell}}(x) = (3p/4)^{\ell}.$ Indeed, $d_{W^{\circ \ell}}(x)$ counts the number of walks from $x$ and if $x \in I_1,$ any such walk stays in $I_1.$ Hence,
    \[ \int_{I_1 \times I_1} W^{\circ \ell} (x, y) \dd x \dd y \le \int_{I_1} d_{W^{\circ \ell}}(x) \dd x = \frac{3^\ell}{4^{\ell+1}} p^\ell. \]
    For $\ell \ge 5,$ we have $\frac{3^\ell}{4^{\ell+1}} p^\ell < \frac{1}{16} p^\ell = |I_1|^2 p^\ell,$ so for $\ell \ge 5,$ $W^{\circ \ell}$ is not $p^\ell$-locally dense.
\end{proof}

Finally, it might be tempting to conjecture that every graph is PSD-nonnegative, which would provide a direction for proving \cref{conj:regular-KNRS}. This however is not the case.

\begin{proposition} \label{prop:psd-nonnegative-counterexample}
    There exists a graph $H$ and a positive-semidefinite $0$-regular kernel $W$ such that $t(H, W) < 0.$
\end{proposition}
\begin{proof}
    Let $H$ be the $6$-vertex graph in Figure~\ref{fig:h-counter} and let $W$ be the kernel defined by the matrix
    \[ \begin{pmatrix}
       18 & -12 & 12 & -12 & -6\\
       -12 & 35 &  28 & -19 & -32\\
       12 & 28 & 56 & -44 & -52\\
       -12 & -19 & -44 & 35 & 40\\
       -6 & -32 & -52 & 40 & 50\\
    \end{pmatrix} \]
    as in the proof of~\cref{prop:counter-examples}. This example was found by computer search. It is straightforward to check that $W$ is a positive semidefinite $0$-regular kernel and that $t(H, W) < 0.$
\end{proof}
\begin{figure}[ht]
    \centering
    \includegraphics[scale=.7]{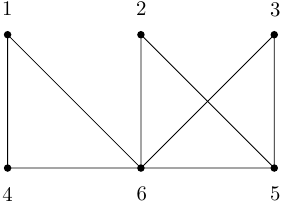}
    \caption{The graph $H$ in \cref{prop:psd-nonnegative-counterexample} which is not PSD-nonnegative.}
    \label{fig:h-counter}
\end{figure}

\appendix
\section{More on graphons}\label{sec:graphon appendix}
\subsection{Graph limits}

As usual, given a function $f:\Omega \to \R$ and a real number $p \geq 1$, we define the $L^p$ norm of $f$ as
\[
\norm f_p \coloneqq \left(\int_\Omega \ab{f(x)}^p\dd x\right)^{1/p},
\]
assuming this integral is finite. We also define $\norm f_\infty$ to be the essential supremum of $\ab f$. These definitions naturally extend to kernels, viewed as functions on the probability space $\Omega \times \Omega$; concretely, we define
\[
\norm W_p \coloneqq \left(\iint_{\Omega \times \Omega} \ab{W(x,y)}^p \dd x \dd y\right)^{1/p},
\]
and similarly $\norm W_\infty$ is the essential supremum of $\ab W$. Note that since kernels are defined to be bounded, $\norm W_p$ is finite for all $1 \leq p \leq \infty$.

More important in the study of graphons is the \emph{cut norm}, which is defined for an arbitrary kernel as
\[
\norm W_\square \coloneqq \sup_{S,T \subseteq \Omega} \bab{\iint_{S \times T} W(x,y) \dd x \dd y},
\]
where the supremum runs over all pairs of measurable subsets $S,T \subseteq \Omega$. It is well-known \cite[Lemma 8.10]{MR3012035} that we can equivalently define the cut norm as
\begin{equation}\label{eq:cut norm functions}
    \norm W_\square = \sup_{f,g:\Omega \to [0,1]} \bab{\iint_{\Omega \times \Omega} f(x) W(x,y) g(y)\dd x \dd y},
\end{equation}
where the supremum runs over all pairs of measurable functions $f,g:\Omega \to [0,1]$.

The cut norm, like any norm, defines a natural metric on the space of kernels, but for the study of graphons it is more useful to define an ``unlabeled'' version of this metric. If $\varphi:\Omega \to \Omega$ is an invertible measure-preserving map, then we denote by $W^\varphi$ the kernel given by $W^\varphi(x,y) \coloneqq W(\varphi(x),\varphi(y))$. The \emph{cut distance} between two kernels $W_1, W_2$ is then defined as
\[
\delta_\square(W_1,W_2)\coloneqq \inf_\varphi \norm{W_1-W_2^\varphi}_\square,
\]
where the infimum runs over all measure-preserving invertible maps $\varphi:\Omega \to \Omega$.

As it turns out, every graph can naturally be viewed as a graphon, as we now define.
\begin{definition} \label{def:W_G}
    Let $G$ be a graph on vertex set $[n]$, and let $\Omega$ be an atomless standard probability space. Arbitrarily partition $\Omega$ into measurable sets $I_1,\dots,I_n$, with $\ab{I_i}=1/n$ for all $i$ (this is possible as $\Omega$ is atomless). For $x \in \Omega$, define $\iota(x)$ to be the unique index $i$ such that $x \in I_i$. We then define a graphon $W_G:\Omega \times \Omega \to [0,1]$ by
    \[
    W_G(x,y) \coloneqq
    \begin{cases}
        1&\text{if }(\iota(x),\iota(y)) \in E(G),\\
        0&\text{otherwise}.
    \end{cases}
    \]
    Note that this definition depends both on the identification between $V(G)$ and $[n]$ and on the choice of partition of $\Omega$, but these differences turn out to be immaterial, and we denote by $W_G$ any graphon that arises in this way.
\end{definition}

We now turn to the notion of graph limits.
\begin{definition}
    Let $G_1,G_2,\dots$ be a sequence of graphs. We say that this sequence \emph{converges} to a graphon $W$ if $\delta_\square(W_{G_n},W) \to 0$ as $n \to \infty$, and we write $G_n \to W$.
\end{definition}
The most important basic result about graphons is that they are the limit objects for sequences of graphs.
\begin{theorem}[{see \cite[Theorem 11.21, Lemma 10.18, Proposition 11.32, and Corollary 11.34]{MR3012035}}]\label{thm:graph limit} 
    If $W$ is a graphon, then there exist graphs $G_1,G_2,\dots$ such that $G_n \to W$.

    Conversely, if $G_1,G_2,\dots$ is an arbitrary sequence of graphs with ${v(G_n)} \to \infty$, then there is a subsequence $G_{n_1},G_{n_2},\dots$ converging to some graphon $W$.
\end{theorem}

 A crucial property about homomorphism densities is that they characterize convergence of graph sequences.
\begin{theorem}[{see e.g.\ \cite[Theorem 4.3.7]{MR4603631}}]\label{thm:left convergence}
    A sequence of graphs $G_1,G_2,\dots$ converges to a graphon $W$ if and only if $t(H,G_n) \to t(H,W)$ for every graph $H$.
\end{theorem}

\subsection{Locally dense graphons}
In this section, we prove various basic results about locally dense graphons, and ultimately prove \cref{lem:graphon equivalence}.

Note that, given a $(p, \delta)$-locally dense graph $G$, its associated graphon $W_G$ is not $p$-locally dense for any $p > 0$ as, for example, the sets corresponding to a single vertex have measure $1 / v(G)$ but the density of $W_G$ inside them is zero. Hence, to make a meaningful connection, we must look at the limit of a sequence of locally dense graphs.

\begin{lemma}\label{lem:graphon LD}
    Let $W$ be a graphon on $\Omega$. $W$ is $p$-locally dense if and only if there exists a sequence of graphs $G_1,G_2,\dots$ with $G_n \to W$ such that $G_n$ is $(p_n,\delta_n)$-locally dense, where $p_n \to p$ and $\delta_n \to 0$ as $n \to \infty$.
\end{lemma}
\begin{proof}
    Suppose first that $W$ is $p$-locally dense. By \cref{thm:graph limit}, there exists a sequence of graphs $(G_n)$ with $G_n \to W$. So it suffices to prove that $G_n$ is $(p_n,\delta_n)$-locally dense, where $p_n=p+o(1)$ and $\delta_n=o(1)$, and all $o(1)$ terms tend to $0$ as $n \to \infty$. 

    Recall that by the definition of convergence,  we have that $\delta_\square(W_{G_n},W)\to 0$. 
    From the definitions of $\delta_\square$ and $W_G$, we see that we may select a partition of $\Omega$ into sets $I_1,\dots,I_{{v(G_n)}}$, each of measure $1/{v(G_n)}$, so that the resulting graphon $W_{G_n}$ satisfies $\norm{W_{G_n}-W}_\square \to 0$. Let $\delta_n = \norm{W_{G_n}-W}_\square^{1/3}$. We claim that $G_n$ is $(p-\delta_n,\delta_n)$-locally dense, which suffices since $\delta_n \to 0$ as $n \to \infty$.
    Suppose for contradiction that $G_n$ is not $(p-\delta_n,\delta_n)$-locally dense; this means that there exists $S \subseteq V(G_n)$ with $\ab S \geq \delta_n {v(G_n)}$, such that $e_{G_n}(S) < (p-\delta_n)\ab S^2/2$.

    Let $U = \bigcup_{i \in S}I_i$. Then from the definition of $W_{G_n}$, we see that 
    \[
    \iint_{U \times U} W_{G_n}(x,y)\dd x \dd y = \frac{2e_{G_n}(S)}{{v(G_n)}^2} < (p-\delta_n) \frac{\ab S^2}{{v(G_n)}^2} = (p-\delta_n)\ab U^2.
    \]
    On the other hand, as $W$ is $p$-locally dense, we have that $\iint_{U \times U}W(x,y)\dd x \dd y \geq p\ab U^2$. Combining these inequalities with the definition of the cut norm, we find that
    \[
         \norm{W_{G_n}-W}_\square \geq \iint_{U \times U} (W-W_{G_n})(x,y)\dd x \dd y > \delta_n \ab U^2 \geq \delta_n^3,
    \]
    which is a contradiction since our choice of $\delta_n$ implies that $\delta_n^3 = \norm{W_{G_n}-W}_\square$.

    It remains to prove the converse, so suppose that $G_n$ is a sequence of $(p_n,\delta_n)$-locally dense graphs, such that $p_n \to p, \delta_n \to 0$, and $G_n \to W$ as $n \to \infty$. As above, we may pick graphons $W_{G_n}$ such that $\norm{W_{G_n}-W}_\square \to 0$ as $n \to \infty$. We may assume that $p > 0,$ otherwise the statement is trivial as any graphon is $0$-locally dense. Note that since $\delta_n \to 0$, we have that $v(G_n) \to \infty$. Suppose for contradiction that $W$ is not $p$-locally dense, and let $U \subseteq \Omega$ be a measurable set such that $\iint_{U \times U}W < p\ab U^2$. Let $\delta=\ab U$, and let $\varepsilon>0$ such that $\iint_{U \times U}W \leq (p-\varepsilon) \ab U^2$. We now define functions $f_n:V(G_n) \to [0,1]$ as follows. Recall that $W_{G_n}$ is obtained by partitioning $\Omega$ into sets $I_1,\dots,I_{{v(G_n)}}$, each of measure $1/{v(G_n)}$. For $1 \leq i \leq {v(G_n)}$, we define $f_n(i) \coloneqq \ab{U \cap I_i}/\ab{I_i}$ to be the fraction of $I_i$ contained in $U$. Note that this definition immediately implies that
    \begin{equation}\label{eq:f property}
        \frac 1{{v(G_n)}^2}\sum_{(i,j) \in E(G_n)} f_n(i)f_n(j) = 2 \iint_{U \times U} W_{G_n}(x,y)\dd x \dd y.
    \end{equation}
    Since $\delta_n \to 0$, we may pick $N_1$ such that $\delta_n\leq \delta$ for all $n \geq N_1$. If $\delta_n \leq \delta=\ab U$, then $\sum_{v \in V(G_n)}f_n(v) = \delta {v(G_n)}\geq \delta_n {v(G_n)}$. Hence, for $n\geq N_1$, \cref{lem:reiher} and \eqref{eq:f property} imply that 
    \[
    \iint_{U \times U} W_{G_n}(x,y) \dd x \dd y \geq p_n \ab U^2 - \frac 1{{v(G_n)}}.
    \]
    However, as $p_n \to p$ and $v(G_n) \to \infty$, we may pick $N_2$ such that for all $n \geq N_2$, we have $p_n |U|^2 - \frac 1{{v(G_n)}} \geq (p - \frac \varepsilon 2) |U|^2$. Thus, for all $n \geq \max\{N_1,N_2\}$, we find that
    \[
    \norm{W_{G_n}-W}_\square \geq \iint_{U \times U} (W_{G_n}-W)(x,y) \dd x \dd y \geq \left(p - \frac{\varepsilon}{2} - (p-\varepsilon)\right) \ab U^2 = \frac \varepsilon 2 \ab U^2.
    \]
    However, the left-hand side tends to $0$ as $n \to \infty$, a contradiction.
\end{proof}

Given \cref{lem:graphon LD}, it is not hard to prove \cref{lem:graphon equivalence}.

\begin{proof}[Proof of \cref{lem:graphon equivalence}]
All three parts are proved in essentially the same way, by combining \cref{thm:graph limit,lem:graphon LD,thm:left convergence}. As such, we omit some details that are repeated.
\begin{enumerate}[label=(\alph*)]
    \item \label{it:KNRS proof}
    Suppose first that $H$ is KNRS. If $W$ is a $p$-locally dense graphon, then by \cref{lem:graphon LD} there exists a sequence of graphs $G_n \to W$ which are $(p+o(1),o(1))$-locally dense. But then the fact that $H$ is KNRS implies that $t(H,G_n) \geq p^{e(H)}-o(1)$. Since $t(H,W) = \lim_n t(H,G_n)$ by \cref{thm:left convergence}, we find that $t(H,W) \geq p^{e(H)}$. For the converse, suppose that $H$ is not KNRS. Then the definition of KNRS fails for some fixed $p,\varepsilon \in (0,1)$. Thus, we may pick a sequence $G_n$ of $(p,o(1))$-locally dense graphs with $t(H,G_n) \leq p^{e(H)}-\varepsilon$. By \cref{thm:graph limit}, there is a convergent subsequence, which necessarily converges to a $p$-locally dense graphon $W$ by \cref{lem:graphon LD}. Hence $t(H,W) \leq \limsup_n t(H,G_n) \leq p^{e(H)}-\varepsilon$, a contradiction.

    \item This is proved in exactly the same way, with the added ingredient that a graphon $W$ is equal to $p$ a.e.\ if and only if it is the limit of a sequence of $(p+o(1),o(1))$-quasirandom graphs \cite[Example 11.37]{MR3012035}.

    \item This is proved in exactly the same way, except now using the fact that if $G_n \to W$, then $t(K_2,G_n) \to t(K_2,W)$ and thus the edge density of $W$ equals the asymptotic edge density of $G_n$.
    \qedhere
\end{enumerate}
\end{proof}

\subsection{Reiher's lemma and new measures}

We begin by proving \cref{lem:reiher graphon}.

\begin{proof}[Proof of \cref{lem:reiher graphon}]
    The ``if'' direction is immediate, for if we set $f$ to be the indicator function of a set $U$, the condition \eqref{eq:reiher graphon} is precisely the statement that $\iint_{U \times U} W \geq p \ab U^2$. So it remains only to prove the ``only if'' direction.

    Let $W$ be a $p$-locally dense graphon, and fix a bounded measurable function $f:\Omega \to [0,\infty)$. Note that \eqref{eq:reiher graphon} is invariant under rescaling of $f$, so we may assume that $f$ has codomain $[0,1]$. There is nothing to prove if $f=0$ a.e.,\ so we may assume that $\int_\Omega f(x) \dd x = \delta>0$. By \cref{lem:graphon LD}, we may pick a sequence of graphs $G_1,G_2,\dots$ such that $G_n$ is $(p_n,\delta_n)$-locally dense, where $p_n \to p$ and $\delta_n \to 0$, such that $\norm{W_{G_n}-W}_\square \to 0$ as $n \to \infty$.

    Recall that $W_{G_n}$ is defined by a partition of $\Omega$ into sets $I_1,\dots,I_{{v(G_n)}}$, each of measure $1/{v(G_n)}$. We define a function $f_n:V(G_n) \to [0,1]$ by setting $f_n(i)$ to be the average value of $f$ on the set $I_i$. Note that 
    \[
    \sum_{i \in V(G_n)} f_n(i) = {v(G_n)} \int_\Omega f(x)\dd x = \delta {v(G_n)}
    \]
    and that
    \[
    \sum_{ij \in E(G)}f_n(i) f_n(j) = \frac 12{v(G_n)}^2\iint_{\Omega \times \Omega}f(x) W_{G_n}(x,y)f(y)\dd x \dd y.
    \]
    Let $N_1$ be chosen so that $\delta_n \leq \delta$ for all $n \geq N_1$. Applying \cref{lem:reiher} to $G_n$ for $n\geq N_1$, we find that
    \begin{align*}
    p_n\norm f_1^2 &= \frac{p_n}{{v(G_n)}^2}\left(\sum_{i \in V(G_n)}f_n(i)\right)^2 
    \leq \frac{2}{{v(G_n)^2}} \sum_{ij \in E(G)} f_n(i) f_n(j) + \frac{2}{{v(G_n)}}\\
    &=\frac{2}{{v(G_n)}} + \iint_{\Omega \times \Omega} f(x) W_{G_n}(x,y)f(y)\dd x \dd y.
    \end{align*}
    On the other hand, by \eqref{eq:cut norm functions}, we have that
    \[
    \bab{\iint_{\Omega \times \Omega}f(x) (W_{G_n}-W)(x,y)f(y)\dd x \dd y} \leq \norm{W_{G_n}-W}_\square \overset{n\to \infty}{\longrightarrow}0.
    \]
    Combining these two estimates, plus the facts that $p_n \to p$ and ${v(G_n)}\to \infty$ as $n\to \infty$, gives the claimed result.
\end{proof}
We remark that we were unable to find a ``direct'' proof of \cref{lem:reiher graphon}, that is, a proof that does not use \cref{lem:reiher}. The difficulty is that Reiher's proof of \cref{lem:reiher} uses the compactness of the unit ball in $\R^n$, whereas the unit ball in $L^2(\Omega)$ is not compact. There are standard ways of overcoming such difficulties (e.g.\ working in the weak topology, where the unit ball is compact by the Banach--Alaoglu theorem), but we were unable to make such techniques work in this context. It would be interesting to find a direct proof of \cref{lem:reiher graphon}.

As mentioned before, \cref{lem:change-measure} is a direct consequence of \cref{lem:reiher graphon}.
\begin{proof}[Proof of~\cref{lem:change-measure}]
    It suffices to prove that for every bounded $f:\Omega \to [0,\infty)$, we have
    \[
    \iint_{\Omega \times \Omega} f(x)W(x,y)f(y)\dd \nu_w(x)\dd\nu_w(y) \geq p \norm f_{1,\nu_w}^2.
    \]
    By the definition of $\nu_w$, the left-hand side equals
    \[
    \iint_{\Omega \times \Omega} w(x)f(x) W(x,y) w(y)f(y)\dd \mu(x)\dd\mu(y),
    \]
    and the right-hand side equals $
    p\left(\int_\Omega w(x)f(x)\dd \mu(x)\right)^2.$
    However, if we define $g=wf$, then $g$ is a bounded function (as both $f$ and $w$ are bounded), so \cref{lem:reiher graphon} implies that
    \begin{align*}
        \iint_{\Omega \times \Omega} w(x)f(x) W(x,y) w(y)f(y)\dd \mu(x)\dd\mu(y)&= \iint_{\Omega \times \Omega} g(x) W(x,y) g(y)\dd \mu(x)\dd \mu(y)\\
        &\geq p \norm g_{1,\mu}^2
        =p \left(\int_\Omega g(x)\dd \mu(x)\right)^2\\
        &=p\left(\int_\Omega w(x)f(x)\dd \mu(x)\right)^2.
    \end{align*}
    This is precisely the claimed inequality.
\end{proof}

As a consequence of \cref{lem:change-measure}, we can prove \cref{lem:count-with-weights}.

\begin{proof}[Proof of \cref{lem:count-with-weights}]
    If $\norm{w}_1 = 0,$ the inequality holds trivially, so we may assume $\norm{w}_1 > 0.$ Define $w'(x) = w(x) / \norm w_1$ and let $\nu$ denote the probability measure associated to $w'.$ By \cref{lem:change-measure}, $W$ is $p$-locally dense with respect to $\nu.$ Since $H$ is KNRS, we have
    \begin{multline*}
        \int_{\Omega^{V(H)}} \prod_{uv \in E(H)} W(x_u, x_v) \prod_{v \in V(H)} w(x_v) \prod_{v \in V(H)} \dd\mu(x_v)=\\
        = \norm w_1^{v(H)} \int_{\Omega^{V(H)}} \prod_{uv \in E(H)} W(x_u, x_v) \prod_{v \in V(H)} \dd\nu(x_v) \ge \norm w_1^{v(H)} p^{e(H)}.        
    \end{multline*}
\end{proof}

\subsection{Kernels as linear and bilinear operators}
In this section we prove the remaining lemmas from \cref{sec:linear bilinear}. We begin with \cref{lem:PSD eigenvalues}.
\begin{proof}[Proof of \cref{lem:PSD eigenvalues}]
    First suppose that $\lambda<0$ for some eigenvalue $\lambda$ of $T_W$, and let $f$ be the corresponding eigenfunction with $\norm f_2=1$. As stated above, $f$ is bounded, and we have that $\inner f{T_W f} = \lambda \norm f_2^2 < 0$, showing that $W$ is not positive semidefinite.

    Conversely, suppose that all eigenvalues of $T_W$ are non-negative, and fix a bounded function $f:\Omega \to \R$; note that this implies $f \in L^2$. Since the eigenfunctions of $T_W$ form an orthonormal basis of $L^2$, we have that
    $\inner f{T_Wf} = \sum_{i=1}^\infty \lambda_i \inner{f}{f_i}^2 \geq 0,$
    hence $W$ is positive semidefinite.
\end{proof}

Next, we prove \cref{lem:LD iff copositive}.
\begin{proof}[Proof of \cref{lem:LD iff copositive}]
    For any bounded non-negative function $f$, we have that
    \[
    \iint_{\Omega \times \Omega} f(x)(W-p)(x,y)f(y)\dd x \dd y = \iint_{\Omega \times \Omega} f(x)W(x,y)f(y) - p\norm f_1^2,
    \]
    since by Fubini's theorem $\iint_{\Omega \times \Omega}f(x)f(y)\dd x \dd y = (\int_\Omega f(x)\dd x)^2$. But by \cref{lem:reiher graphon}, $W$ is $p$-locally dense if and only if $\iint_{\Omega \times \Omega} f(x) W(x,y)f(y)\dd x \dd y \geq p\norm f_1^2$ for all such functions, which by the above is equivalent to the statement that $W-p$ is copositive.
\end{proof}

Finally, we prove \cref{lem:hadamard PSD}. We remark that this result is fairly easy and well-known in case the kernels $W_1,W_2$ are assumed to be continuous functions $[0,1]^2 \to [0,1]$ (in this case it follows from Mercer's theorem, or from an appropriate infinitary analogue of the argument that the Hadamard product of two matrices is a principal submatrix of their tensor product). However, as we do not assume any continuity, the proof becomes somewhat lengthier.

\begin{proof}[Proof of \cref{lem:hadamard PSD}]
    We write the spectral decompositions of $W_1,W_2$ as
    \[
    W_1(x,y) \sim \sum_{i=1}^\infty \lambda_i f_i(x)f_i(y), \qquad W_2(x,y) \sim \sum_{j=1}^\infty \mu_j g_j(x) g_j(y)
    \]
    where $\sim$ denotes convergence in $L^2(\Omega \times \Omega)$. In these spectral decompositions, we may assume that $\lambda_i,\mu_j \neq0$ for all $i,j$ (by simply ommitting the terms corresponding to zero eigenvalues), and hence we may assume that $f_i,g_j$ are bounded for all $i,j$. Note that, by omitting the zero eigenvalues, we may no longer assume that $\{f_i\}$ forms an orthonormal basis of $L^2(\Omega)$ (as the eigenfunctions corresponding to zero eigenvalues are not spanned), but luckily we will not need them to form an orthonormal basis. Note that, by \cref{lem:PSD eigenvalues}, we have that $\lambda_i, \mu_j>0$ for all $i,j$.

    Define $W=W_1 \odot W_2$. For positive integers $K,M$, let
    
    \[
    W_1^K(x,y) \coloneqq \sum_{i=1}^K \lambda_i f_i(x) f_i(y) \qquad\text{and}\qquad W_2^M \coloneqq \sum_{j=1}^M \mu_j g_j(x) g_j(y).
    \]
    Let $W^{KM}\coloneqq W_1^K \odot W_2^M$, and note that we may write
    \[
    W^{KM}(x,y) = \sum_{i=1}^K \sum_{j=1}^M \lambda_i \mu_j h_{ij}(x) h_{ij}(y),
    \]
    where $h_{ij}(x)\coloneqq f_i(x)g_j(x)$ is a bounded measurable function $\Omega \to \R$. In particular, this decomposition shows that $W^{KM}$ is positive-semidefinite for all $K,M$, since for any bounded $a:\Omega \to \R$, we have
    \begin{align*}
    \iint_{\Omega \times \Omega}a(x) W^{KM} (x,y) a(y)\dd x \dd y &= \sum_{i=1}^K \sum_{j=1}^M \lambda_i \mu_j \iint_{\Omega \times \Omega} a(x) h_{ij}(x) a(y) h_{ij}(y)\dd x \dd y\\
    &=\sum_{i=1}^K\sum_{j=1}^M \lambda_i \mu_j \left(\int_\Omega a(x) h_{ij}(x)\dd x\right)^2 \geq 0,
    \end{align*}
    where the final inequality holds since every summand is non-negative. We also note here for future reference that the $L^2$-limit of positive semidefinite kernels is positive semidefinite. Indeed, suppose that $X_1,X_2,\dots,X$ are kernels with $\norm{X-X_k}_{L^2(\Omega \times \Omega)} \to 0$, and suppose that each $X_k$ is positive semidefinite. It is well-known (e.g.\ \cite[Proposition 5.5(ii)]{MR2129625}) that the $L^2 \to L^2$ operator norm of $T_X$ is upper-bounded by $\norm X_{L^2 (\Omega \times \Omega)}$. Applying this to $X-X_k$, as well as Cauchy--Schwarz, implies that for any bounded $a:\Omega \to \R$, we have
    \begin{align*}
    \ab{\inner{a}{T_X a}-\inner{a}{T_{X_k}a}}&=\ab{\inner{a}{T_{X-X_k}a}}
    \leq \norm a_2 \norm{T_{X-X_k}a}_2 
    \leq \norm a_2^2 \norm{T_{X-X_k}}_{L^2 \to L^2} \\
    &\leq \norm a_2^2 \norm{X-X_k}_{L^2(\Omega \times \Omega)} \overset{k \to \infty}{\longrightarrow}0.
    \end{align*}
    In other words, $\inner a{T_{X_k}a} \to \inner a{T_X a}$ as $k \to \infty$, implying that $\inner a{T_X a}\geq 0$, and hence that $X$ is positive semidefinite. This shows, as claimed,  that the limit of positive semidefinite kernels is positive semidefinite; we will shortly use this observation twice.

    For a fixed $K$, we now claim that $W^{KM} \to W_1^K \odot W_2$ as $M \to \infty$, where the convergence is in $L^2(\Omega \times \Omega)$. Indeed, we have that
    \begin{align*}
        \norm{W^{KM}-W_1^K \odot W_2}_2 = \norm{W_1^K(W_2^M-W_2)}_2 \leq \norm{W_1^K}_\infty \norm{W_2^M - W_2}_2.
    \end{align*}
    Since $W_1^K$ is a finite sum of bounded functions, it is bounded, so $\norm {W_1^K}_\infty$ is some finite number depending only on $K$. However, as $M\to \infty$, we have that $\norm{W_2^M-W_2}_2 \to 0$, by the definition of $W_2^M$. This shows that, for any fixed $K$, we have $W^{KM} \to W_1^K \odot W_2$ in $L^2$. In particular, as we showed that each $W^{KM}$ is positive semidefinite, we conclude that $W_1^K \odot W_2$ is positive semidefinite for all $K$.

    We now argue in almost the same way that $W_1^K \odot W_2 \to W$ in $L^2$; as above, this implies that $W$ is positive semidefinite, which is what we wanted to prove. So it remains to prove the convergence, which holds since
    \[
    \norm{W_1^K \odot W_2 - W}_2 = \norm{W_2 (W_1^K - W_1)}_2 \leq \norm{W_2}_\infty \norm{W_1^K - W_1}_2,
    \]
    and we know that $\norm{W_2}_\infty$ is finite as $W_2$ is a kernel, and that $W_1^K \to W_1$ in $L^2$ by the definition of $W_1^K$.
\end{proof}


\begin{thebibliography}{10}
\providecommand{\url}[1]{\texttt{#1}}
\providecommand{\urlprefix}{URL }
\providecommand{\eprint}[2][]{\url{#2}}

\bibitem{MR3425981}
O.~Antol\'{\i}n~Camarena, E.~Cs\'{o}ka, T.~Hubai, G.~Lippner, and L.~Lov\'{a}sz, Positive graphs, \emph{European J. Combin.} \textbf{52} (2016), 290--301.

\bibitem{MR0184950}
G.~R. Blakley and P.~Roy, A {H}\"{o}lder type inequality for symmetric matrices with nonnegative entries, \emph{Proc. Amer. Math. Soc.} \textbf{16} (1965), 1244--1245.

\bibitem{MR4608432}
M.~Buci\'{c}, J.~W. Cooper, D.~Kr\'{a}\v{l}, S.~Mohr, and D.~Munh\'{a}~Correia, Uniform {T}ur\'{a}n density of cycles, \emph{Trans. Amer. Math. Soc.} \textbf{376} (2023), 4765--4809.

\bibitem{chung1989quasi}
F.~R.~K. Chung, R.~L. Graham, and R.~M. Wilson, Quasi-random graphs, \emph{Combinatorica} \textbf{9} (1989), 345--362.

\bibitem{MR2738996}
D.~Conlon, J.~Fox, and B.~Sudakov, An approximate version of {S}idorenko's conjecture, \emph{Geom. Funct. Anal.} \textbf{20} (2010), 1354--1366.

\bibitem{MR2864650}
D.~Conlon, H.~H\`an, Y.~Person, and M.~Schacht, Weak quasi-randomness for uniform hypergraphs, \emph{Random Structures Algorithms} \textbf{40} (2012), 1--38.

\bibitem{MR3893193}
D.~Conlon, J.~H. Kim, C.~Lee, and J.~Lee, Some advances on {S}idorenko's conjecture, \emph{J. Lond. Math. Soc. (2)} \textbf{98} (2018), 593--608.

\bibitem{MR4237083}
D.~Conlon and J.~Lee, Sidorenko's conjecture for blow-ups, \emph{Discrete Anal.}  (2021), Paper No. 2, 13.

\bibitem{CoLeVe}
D.~Conlon, J.~Lee, and L.~Versteegen, Around the positive graph conjecture, 2024. Preprint available at arXiv:2404.17467.

\bibitem{CoRa}
L.~N. Coregliano and A.~A. Razborov, Biregularity in {S}idorenko's conjecture, 2021. Preprint available at arXiv:2108.06599.

\bibitem{MR0205876}
P.~Erd\H{o}s and M.~Simonovits, A limit theorem in graph theory, \emph{Studia Sci. Math. Hungar.} \textbf{1} (1966), 51--57.

\bibitem{MR0698654}
P.~Erd\H{o}s and V.~T. S\'{o}s, On {R}amsey-{T}ur\'{a}n type theorems for hypergraphs, \emph{Combinatorica} \textbf{2} (1982), 289--295.

\bibitem{MR0018807}
P.~Erd\"{o}s and A.~H. Stone, On the structure of linear graphs, \emph{Bull. Amer. Math. Soc.} \textbf{52} (1946), 1087--1091.

\bibitem{MR3474967}
R.~Glebov, D.~Kr\'{a}l', and J.~Volec, A problem of {Erd\H{o}s} and {S}\'{o}s on 3-graphs, \emph{Israel J. Math.} \textbf{211} (2016), 349--366.

\bibitem{MR2607540}
H.~Hatami, Graph norms and {S}idorenko's conjecture, \emph{Israel J. Math.} \textbf{175} (2010), 125--150.

\bibitem{MR3456171}
J.~H. Kim, C.~Lee, and J.~Lee, Two approaches to {S}idorenko's conjecture, \emph{Trans. Amer. Math. Soc.} \textbf{368} (2016), 5057--5074.

\bibitem{MR2595699}
Y.~Kohayakawa, B.~Nagle, V.~R\"{o}dl, and M.~Schacht, Weak hypergraph regularity and linear hypergraphs, \emph{J. Combin. Theory Ser. B} \textbf{100} (2010), 151--160.

\bibitem{MR4201799}
J.~Lee, On some graph densities in locally dense graphs, \emph{Random Structures Algorithms} \textbf{58} (2021), 322--344.

\bibitem{MR3012035}
L.~Lov\'{a}sz, \emph{Large networks and graph limits}, \emph{American Mathematical Society Colloquium Publications}, vol.~60, American Mathematical Society, Providence, RI, 2012.

\bibitem{MR2737279}
V.~Nikiforov, The number of cliques in graphs of given order and size, \emph{Trans. Amer. Math. Soc.} \textbf{363} (2011), 1599--1618.

\bibitem{MR2433944}
A.~A. Razborov, On the minimal density of triangles in graphs, \emph{Combin. Probab. Comput.} \textbf{17} (2008), 603--618.

\bibitem{MR0751959}
M.~Reed and B.~Simon, \emph{Methods of modern mathematical physics. {I}: {F}unctional analysis}, second ed., Academic Press, Inc. [Harcourt Brace Jovanovich, Publishers], New York, 1980.

\bibitem{MR3171777}
C.~Reiher, Counting odd cycles in locally dense graphs, \emph{J. Combin. Theory Ser. B} \textbf{105} (2014), 1--5.

\bibitem{MR3549620}
C.~Reiher, The clique density theorem, \emph{Ann. of Math. (2)} \textbf{184} (2016), 683--707.

\bibitem{MR4111729}
C.~Reiher, Extremal problems in uniformly dense hypergraphs, \emph{European J. Combin.} \textbf{88} (2020), 103117, 22.

\bibitem{MR3848224}
C.~Reiher, V.~R\"{o}dl, and M.~Schacht, On a generalisation of {M}antel's theorem to uniformly dense hypergraphs, \emph{Int. Math. Res. Not. IMRN}  (2018), 4899--4941.

\bibitem{MR3790065}
C.~Reiher, V.~R\"{o}dl, and M.~Schacht, On a {T}ur\'{a}n problem in weakly quasirandom 3-uniform hypergraphs, \emph{J. Eur. Math. Soc. (JEMS)} \textbf{20} (2018), 1139--1159.

\bibitem{MR1276825}
V.~R\"{o}dl and A.~Ruci\'{n}ski, Threshold functions for {R}amsey properties, \emph{J. Amer. Math. Soc.} \textbf{8} (1995), 917--942.

\bibitem{MR3666677}
V.~R\"{o}dl, A.~Ruci\'{n}ski, and M.~Schacht, Ramsey properties of random graphs and {F}olkman numbers, \emph{Discuss. Math. Graph Theory} \textbf{37} (2017), 755--776.

\bibitem{sidorenko1993correlation}
A.~Sidorenko, A correlation inequality for bipartite graphs, \emph{Graphs Combin.} \textbf{9} (1993), 201--204.

\bibitem{MR4444712}
A.~Sidorenko, On positive hypergraphs, \emph{European J. Combin.} \textbf{106} (2022), Paper No. 103574, 5.

\bibitem{MR2080111}
J.~Skokan and L.~Thoma, Bipartite subgraphs and quasi-randomness, \emph{Graphs Combin.} \textbf{20} (2004), 255--262.

\bibitem{MR2129625}
E.~M. Stein and R.~Shakarchi, \emph{Real analysis: {M}easure theory, integration, and {H}ilbert spaces}, \emph{Princeton Lectures in Analysis}, vol.~3, Princeton University Press, Princeton, NJ, 2005.

\bibitem{Szegedy}
B.~Szegedy, Sparse graph limits, entropy maximization and transitive graphs, 2015. Preprint available at arXiv:1504.00858.

\bibitem{MR0930498}
A.~Thomason, Pseudorandom graphs, in \emph{Random graphs '85 ({P}ozna\'{n}, 1985)}, \emph{North-Holland Math. Stud.}, vol. 144, North-Holland, Amsterdam, 1987,  307--331.

\bibitem{MR4603631}
Y.~Zhao, \emph{Graph theory and additive combinatorics---exploring structure and randomness}, Cambridge University Press, Cambridge, 2023.

\end{thebibliography}
\end{document}